\documentclass[twoside,12pt,reqno]{amsart}
\usepackage{amsmath,amscd,amssymb,epsfig,subfig}
\usepackage[all]{xy}
\usepackage{verbatim}
\usepackage{graphics,graphicx}
\usepackage{hyperref}

\newtheorem{definition}{Definition}
\newtheorem{theorem}[definition]{Theorem}

\newtheorem{proposition}[definition]{Proposition}
\newtheorem{lemma}[definition]{Lemma}
\newtheorem{remark}[definition]{Remark}
\newtheorem{example}[definition]{Example}
\newtheorem{corollary}[definition]{Corollary}

\newcommand{\C}{\ensuremath{\mathbb{C}} }
\newcommand{\Z}{\ensuremath{\mathbb{Z}} }
\newcommand{\R}{\ensuremath{\mathbb{R}} }

\newcommand{\g}{\ensuremath{\mathfrak{g}}}
\newcommand{\slt}{\ensuremath{\mathfrak{sl}(2)}}
\newcommand{\Uslt}{\ensuremath{U_{q}(\slt) } }
\newcommand{\Uqg}{\ensuremath{U_{q}(\g) } }
\newcommand{\UsltH}{\ensuremath{U^H_{q}(\slt) } }
\newcommand{\Uqghat}{\ensuremath{\widehat{U}_{q}(\g) } }
\newcommand{\Ubar}{\bar{U}^H_{q}(\slt)}
\newcommand{\Utilde}{\bar{U}_{q}(\slt)}

\newcommand{\End}{\operatorname{End}}
\newcommand{\Hom}{\operatorname{Hom}}
\newcommand{\Int}{\operatorname{Int}}
\newcommand{\tr}{\operatorname{tr}}
\newcommand{\Id}{\operatorname{Id}}

\newcommand{\wb}{\overline}
\newcommand{\vect}{\overrightarrow}

\newcommand{\cntr}{{\operatorname {cntr}}}

\newcommand{\e}{\operatorname{e}}

\newcommand{\unit}{\ensuremath{\mathbb{I}}}
\newcommand{\cat}{\mathcal{C}}

\newcommand{\catd}{{\mathcal{C}^H}}

\newcommand{\Graph}{\operatorname{Gr}}
\newcommand{\tet}{\mathcal{T}}
\newcommand{\FK}{K}
\newcommand{\LL}{\mathcal{L}}
\newcommand{\Gr}{G}
\newcommand{\Gs}{X}
\renewcommand{\wp}{{\Phi}}
\newcommand{\p}{\varphi}

\newcommand{\qn}[1]{{\left\{#1\right\}}}

\newcommand{\charge}{height}

\newcommand{\qd}{\operatorname{\mathsf{d}}}
\newcommand{\bb}{\operatorname{\mathsf{b}}}
\newcommand{\ro}{r} 
\newcommand{\m}{{\ro'}}
\newcommand{\states}{\operatorname{St}}
\newcommand{\twist}{\theta}
\newcommand{\bubble}{$H$-bubble}
\newcommand{\Pachner}{$H$-Pachner}
\newcommand{\lune}{$H$-lune}

\newcommand{\mathsmall}[1]{\mbox{\tiny$#1$}}
\newcommand{\mathsmal}[1]{\mbox{\small$#1$}}
\newcommand{\sjtop}[6]{\left|\begin{array}{ccc}#1 & #2 & #3 \\#4 & #5 &
      #6\end{array}\right|}
\newcommand{\sj}[6]{\left\{\begin{array}{ccc}#1 & #2 & #3 \\#4 & #5 &
      #6\end{array}\right\}}
\newcommand{\sjs}[6]{\left\{\begin{array}{ccc}#1 & #2 & #3 \\#4 & #5 &
      #6\end{array}\right\}^\sigma}
\newcommand{\epsh}[2]
         {\begin{array}{c} \hspace{-1.3mm}
        \raisebox{-4pt}{\epsfig{figure=#1,height=#2}}
        \hspace{-1.9mm}\end{array}}

       \newcommand{\T}{{\mathcal{T}}}
\newcommand{\pic}[2]{
  \setlength{\unitlength}{#1}
  {\begin{array}{c} \hspace{-1.3mm}
        \raisebox{-4pt}{#2}
        \hspace{-1.9mm}\end{array}}}

\begin{document}
\title[Modified $6j$-Symbols and $3$-Manifold Invariants]{Modified
  $6j$-Symbols and $3$-Manifold Invariants}
\author{Nathan Geer}
\address{Max-Planck-Institut f\"{u}r Mathematik\\
  Vivatsgasse 7\\
  53111 Bonn, Germany\\
  and Mathematics \& Statistics\\
  Utah State University \\
  Logan, Utah 84322, USA} \email{nathan.geer@usu.edu} 
\author{Bertrand Patureau-Mirand}
\address{L.M.A.M., - Universit\'e  Europ\'eenne de Bretagne\\
  Universit\'e de Bretagne-Sud, BP 573\\
  F-56017 Vannes, France } \email{bertrand.patureau@univ-ubs.fr}
\author{Vladimir Turaev}
\address{Department of Mathematics \\
  Indiana University \\
  Rawles Hall, 831 East 3rd St \\
  Bloomington, IN 47405, USA}
\date{\today} \thanks{The work of N.\ Geer was partially supported by the NSF
  grant DMS-0706725. The work of V.\ Turaev was partially supported by the NSF
  grants DMS-0707078 and DMS-0904262.  }

\begin{abstract}
  We show that the renormalized quantum invariants of links and graphs in the
  3-sphere, derived from tensor categories in \cite{GPT}, lead to modified
  $6j$-symbols and to new state sum $3$-manifold invariants.  We give examples
  of categories such that the associated standard Turaev-Viro $3$-manifold
  invariants vanish but the secondary invariants may be non-zero.  The
  categories in these examples are pivotal categories which are neither ribbon
  nor semi-simple and have an infinite number of simple objects.
\end{abstract}

\maketitle

{\it Dedicated to Jose Maria Montesinos on the  occasion of his 65th
birthday}

\section*{Introduction}

The numerical $6j$-symbols associated with the Lie algebra $sl_2(\C)$ were
first introduced in theoretical physics by Eugene Wigner in 1940 and Giulio
(Yoel) Racah in 1942. They were extensively studied in the quantum theory of
angular momentum, see for instance \cite{ed}. In mathematics, $6j$-symbols
naturally arise in the study of semisimple monoidal tensor categories. In this
context, the $6j$-symbols are not numbers but rather tensors on 4 variables
running over certain multiplicity spaces. The $6j$-symbols have interesting
topological applications: one can use them to write down state sums on knot
diagrams and on triangulations of 3-manifolds yielding topological invariants
of knots and of 3-manifolds, see \cite{KRR}, \cite{TV}, \cite{BW+}.

The aim of this paper is to introduce and to study \lq\lq modified"
$6j$-symbols. This line of research extends our previous paper \cite{GPT}
where we introduced so-called modified quantum dimensions of objects of a
monoidal tensor category. These new dimensions are particularly interesting
when the usual quantum dimensions are zero since the modified dimensions may
be non-zero. The standard $6j$-symbols can be viewed as a far-reaching
generalization of quantum dimensions of objects. Therefore it is natural to
attempt to \lq\lq modify" their definition following the ideas of \cite{GPT}.
We discuss here a natural setting for such a modification based on a notion of
an ambidextrous pair. In this setting, our constructions produce an
interesting (and new) system of tensors. These tensors share some of the
properties of the standard $6j$-symbols including the fundamental
Biedenharn-Elliott identity. We suggest a general scheme allowing to derive
state-sum topological invariants of links in 3-manifolds from the modified
$6j$-symbols.

As an example, we study the modified $6j$-symbols associated with a version of
the category of representations of $sl_2(\C)$. More precisely, we introduce a
Hopf algebra $\Utilde$ and study the monoidal category of weight
$\Utilde$-modules at a complex root of unity $q$ of odd order $r$. This
category has an infinite number of simple objects parametrized by elements of
$\C/2r \Z\approx \C^*$.  The standard quantum dimensions of these objects and
the standard $6j$-symbols are generically zero. We show that the modified
quantum dimensions are generically non-zero.  As an application, we construct
state-sum topological invariants of triples (a closed oriented 3-manifold $M$,
a link in $M$, an element of $H^1(M; \C^*)$).

Our techniques are similar to those used by Kashaev \cite{Kas2} and later
Baseilhac and Benedetti \cite{BB} who derived $3$-manifold invariants from the
category of modules over the Borel subalgebra of $U_q(sl_2(\C))$ at a root of
unity $q$. It is interesting to understand exact connections between these
constructions; we are currently working on this question.


\section{Pivotal tensor  categories}\label{S:SphCat}

We recall the definition of a pivotal tensor category, see for instance,
\cite{BW}. A \emph{tensor category} $\cat$ is a category equipped with a
covariant bifunctor $\otimes :\cat \times \cat\rightarrow \cat$ called the
tensor product, an associativity constraint, a unit object $\unit$, and left
and right unit constraints such that the Triangle and Pentagon Axioms hold.
When the associativity constraint and the left and right unit constraints are
all identities we say that $\cat$ is a \emph{strict} tensor category. By
MacLane's coherence theorem, any tensor category is equivalent to a strict
tensor category. To simplify the exposition, we formulate further definitions
only for strict tensor categories; the reader will easily extend them to
arbitrary tensor categories.

A strict tensor category $\cat$ has a \emph{left duality} if for each object
$V$ of $\cat$ there are an object $V^*$ of $\cat$ and morphisms
\begin{equation}\label{lele}
  b_{V} : \:\:   \unit \rightarrow V\otimes V^{*} \quad {\rm {and}} \quad
   d_{V}: \:\:
  V^*\otimes V\rightarrow \unit
\end{equation}
such that
\begin{align*}
  (\Id_V\otimes d_V)(b_V \otimes \Id_V)&=\Id_V & & {\rm {and}} & (d_V\otimes
  \Id_{V^*})(\Id_{V^*}\otimes b_V)&=\Id_{V^*}.
\end{align*}
A left duality determines for every morphism $f:V\to W$ in $\cat$ the dual (or 
transposed) morphism $f^*:W^*\rightarrow V^*$ by
$$
f^*=(d_W \otimes \Id_{V^*})(\Id_{W^*} \otimes f \otimes
\Id_{V^*})(\Id_{W^*}\otimes b_V),
$$
and determines for any objects $V,W$ of $\cat$, an isomorphism
$\gamma_{V,W}: W^*\otimes V^* \rightarrow (V\otimes W)^*$ by
$$
\gamma_{V,W} = (d_W\otimes \Id_{(V\otimes W)^*})(\Id_{W^*} \otimes d_V \otimes
\Id_W \otimes \Id_{(V\otimes W)^*})(\Id_{W^*}\otimes \Id_{V^*} \otimes
b_{V\otimes W}).
$$

Similarly, $\cat$ has a \emph{right duality} if for each object $V$ of $\cat$
there are an object $ V^\bullet$ of $\cat$ and morphisms
\begin{equation}\label{roro}
  b'_{V} : \:\:   \unit\rightarrow V^\bullet\otimes V \quad {\rm {and}} \quad
  d_{V}':\:\:   V\otimes V^\bullet\rightarrow \unit
\end{equation}
such that
\begin{align*}
  (\Id_{V^\bullet}\otimes d'_V)(b'_V \otimes \Id_{V^\bullet})&=\Id_{V^\bullet}
  & & {\rm {and}} & (d'_V\otimes \Id_{V})(\Id_{V}\otimes b'_V)&=\Id_{V}.
\end{align*}
The right duality determines for every morphism $f:V\to W$ in $\cat$ the dual
morphism $f^\bullet:W^\bullet\rightarrow V^\bullet$ by
$$
f^\bullet=(\Id_{V^\bullet} \otimes d'_W ) (\Id_{V^\bullet} \otimes f \otimes
\Id_{W^\bullet})( b'_V \otimes \Id_{W^\bullet}),
$$
and determines for any objects $V,W$, an isomorphism $\gamma'_{V,W}:
W^\bullet\otimes V^\bullet \rightarrow (V\otimes W)^\bullet$ by
$$
\gamma'_{V,W}
= ( \Id_{(V\otimes W)^\bullet} \otimes d'_V )(\Id_{(V\otimes W)^\bullet}
\otimes \Id_{V} \otimes d'_W \otimes \Id_{V^\bullet} )(b'_{V\otimes W}\otimes
\Id_{W^\bullet}\otimes \Id_{V^\bullet}).
$$

A \emph{pivotal category} is a tensor category with left duality $\{{b_V},
d_V\}_V $ and right duality $\{{b'_V}, d'_V\}_V $ which are compatible in the
sense that $V^*=V^\bullet$, $f^*=f^\bullet$, and $\gamma_{V,W}=\gamma'_{V,W}$
for all $V, W, f$ as above.

A tensor category $\cat$ is said to be an \emph{Ab-category} if for any
objects $V,W$ of $\cat$, the set of morphism $\Hom(V,W)$ is an additive
abelian group and both composition and tensor multiplication of morphisms are
bilinear.  Composition of morphisms induces a commutative ring structure on
the abelian group $K=\End(\unit)$. The resulting ring is called the
\emph{ground ring} of $\cat$. For any objects $V,W$ of $\cat$ the abelian
group $\Hom(V,W)$ becomes a left $K$-module via $kf=k\otimes f$ for $k\in K$
and $f\in \Hom(V,W)$. An object $V$ of $\cat$ is \emph{simple} if $\End(V)= K
\Id_V$.

\section{Invariants    of   graphs in $S^2$}\label{S:InvOfGraphs}

From now on and up to the end of the paper the symbol $\cat$ denotes
a pivotal tensor Ab-category with ground ring $K$, left duality
 \eqref{lele}, and right duality \eqref{roro}.

\begin{figure}[b]
  \centering $ \xymatrix{ \ar@< 8pt>[d]^{W_m}_{... \hspace{1pt}}
    \ar@< -8pt>[d]_{W_1}\\
    *+[F]\txt{ \: \; f \; \;} \ar@< 8pt>[d]^{V_n}_{... \hspace{1pt}}
    \ar@< -8pt>[d]_{V_1}\\
    \: } $
 \caption{}\label{F:Coupon}
\end{figure}
A morphism $f:V_1\otimes{\cdots}\otimes V_n \rightarrow
W_1\otimes{\cdots}\otimes W_m$ in $\cat$ can be represented by a box and
arrows as in Figure \ref{F:Coupon}.  Here the plane of the picture is oriented
counterclockwise, and this orientation determines the numeration of the arrows
attached to the bottom and top sides of the box.  More generally, we allow
such boxes with some arrows directed up and some arrows directed down. For
example, if all the bottom arrows in the above picture and redirected upward,
then the box represents a morphism $V_1^*\otimes{\cdots}\otimes V_n^*
\rightarrow W_1\otimes{\cdots}\otimes W_m$. The boxes as above are called {\it
  coupons}.  Each coupon has distinguished bottom and top sides and all
incoming and outgoing arrows can be attached only to these sides.

By a graph we always mean a finite graph with oriented edges (we allow loops
and multiple edges with the same vertices).  By a $\cat$-colored ribbon graph
in an oriented surface $\Sigma$, we mean a graph embedded in $\Sigma$ whose
edges are colored by objects of $\cat$ and whose vertices lying in $\Int
\Sigma=\Sigma-\partial \Sigma $ are thickened to coupons colored by morphisms
of~$\cat$.  The edges of a ribbon graph do not meet each other and may meet
the coupons only at the bottom and top sides.  The intersection of a
$\cat$-colored ribbon graph in $\Sigma$ with $\partial \Sigma$ is required to
be empty or to consist only of vertices of valency~1.

We define a category of $\cat$-colored ribbon graphs $\Graph_\cat$.  The
objects of $\Graph_\cat$ are finite sequences of pairs $(V,\varepsilon)$,
where $V$ is an object of $\cat$ and $\varepsilon=\pm$. The morphisms of
$\Graph_\cat$ are isotopy classes of $\cat$-colored ribbon graphs $\Gamma$
embedded in $\R\times [0,1]$. The (1-valent) vertices of such $\Gamma$ lying
on $\R\times 0$ are called the inputs. The colors and orientations of the
edges of $\Gamma$ incident to the inputs (enumerated from the left to the
right) determine an object of $\Graph_\cat$ called the source of
$\Gamma$. Similarly, the (1-valent) vertices of $\Gamma$ lying on $\R\times 1$
are called the outputs; the colors and orientations of the edges of $\Gamma$
incident to the outputs determine an object of $\Graph_\cat$ called the target
of $\Gamma$.  We view $\Gamma$ as a morphism in $\Graph_\cat$ from the source
object to the target object.  More generally, formal linear combinations over
$K$ of $\cat$-colored ribbon graphs in $\R\times [0,1]$ with the same input
and source are also viewed as morphisms in $\Graph_\cat$. Composition, tensor
multiplication, and left and right duality in $\Graph_\cat$ are defined in the
standard way, cf.  \cite{Tu}. This makes $\Graph_\cat$ into a pivotal
Ab-category.

The well-known Reshetikhin-Turaev construction defines a $K$-linear functor $G
: \Graph_\cat\rightarrow \cat$ preserving both left and right dualities. This
functor is compatible with tensor multiplication and transforms an object $(V,
\varepsilon)$ of $\Graph_\cat$ to $V$ if $\varepsilon=+$ and to $V^*$ if
$\varepsilon=-$.  The definition of $G$ goes by splitting the ribbon graphs
into ``elementary'' pieces, cf.\ \cite{Tu}. The left duality in $\cat$ is used
to assign morphisms in $\cat$ to the small right-oriented cup-like and
cap-like arcs.  The right duality is used to assign morphisms in $\cat$ to the
small left-oriented cup-like and cap-like arcs. Invariance of $G$ under plane
isotopy is deduced from the conditions $f^*= f^\bullet$ and $\gamma_{V,W}=
\gamma'_{V,W}$ in the definition of a pivotal tensor category.

Under certain assumptions, one can thicken the usual (non-ribbon) graphs in
$S^2$ into ribbon graphs.  The difficulty is that thickening of a vertex to a
coupon is not unique. We recall a version of thickening for trivalent graphs
following~\cite{Tu}.

By {\it basic data} in $\cat $ we mean a family $\{V_i\}_{i\in I}$ of simple
objects of $\cat$ numerated by elements of a set $I$ with involution
$I\rightarrow I$, $i\mapsto i^*$ and a family of isomorphisms $\{w_i:V_i \to
V_{i^*}^*\}_{i \in I}$ such that $\Hom_\cat (V_i, V_j)=0$ for distinct $i,j\in
I$ and
\begin{equation}\label{E:FamilyIso}
  d_{V_i}(w_{i^*} \otimes \Id_{V_i})=d'_{V_{i^*}}(\Id_{V_{i^*}}\otimes
  w_i)\colon V_{i^*} \otimes V_{i   }\to \unit 
\end{equation}
for all $i\in I$. In particular, $V_i$ is not isomorphic to $V_j$ for $i\neq
j$.

For any $i,j,k\in I$, consider the multiplicity module
$$
H^{ijk}=\Hom(\unit, V_i\otimes V_j \otimes V_k).
$$
The $K$-modules $H^{ijk},H^{jki},H^{kij}$ are canonically isomorphic. Indeed,
let $\sigma(i,j,k)$ be the isomorphism
$$
H^{ijk} \rightarrow H^{jki},\,\,\,
x \mapsto d_{V_i}\circ (\Id_{V_{i}^*}\otimes x \otimes \Id_{V_i})\circ
b'_{V_i}.
$$
Using the functor $G : \Graph_\cat\rightarrow \cat$, one easily proves that
$$
\sigma(k,i,j) \, \sigma(j,k,i)\, \sigma(i,j,k) =\Id_{H^{ijk}}.
$$
\noindent Identifying the modules $H^{ijk},H^{jki},H^{kij}$ along these
isomorphisms we obtain a {\it symmetrized multiplicity module} $H(i,j,k)$
depending only on the cyclically ordered triple $(i,j,k)$. This construction
may be somewhat disturbing, especially if some (or all) of the indices $i,j,k$
are equal. We give therefore a more formal version of the same
construction. Consider a set $X$ consisting of three (distinct) elements
$\{a,b,c\}$ with cyclic order $a<b<c<a$. For any function $f:X\to I$, the
construction above yields canonical $K$-isomorphisms
$$
H^{f(a), f(b), f(c)} \cong H^{f(b), f(c), f(a)} \cong H^{f(c), f(a), f(b)}
$$
of the modules determined by the linear orders on $X$ compatible with the
cyclic order. Identifying these modules along these isomorphisms we obtain a
module $H(f)$ independent of the choice of a linear order on $X$ compatible
with the cyclic order.  If $i=f(a), j= f(b), k=f(c)$, then we write $H(i,j,k)$
for $H(f)$.

By a labeling of a graph we mean a function assigning to every edge of the
graph an element of $I$. By a trivalent graph we mean a (finite oriented)
graph whose vertices all have valency $3$. Let $\Gamma$ be a labeled trivalent
graph in $S^2$.  Using the standard orientation of $S^2$ (induced by the
right-handed orientation of the unit ball in $ \R^3$), we cyclically order the
set $X_v$ of 3 half-edges adjacent to any given vertex $v$ of $\Gamma$.  The
labels of the edges determine a function $f_v:X_v\to I$ as follows: if a
half-edge $e$ adjacent to $v$ is oriented towards $v$, then $f_v(e)=i$ is the
label of the edge of $\Gamma$ containing $e$; if a half-edge $e$ adjacent to
$v$ is oriented away from $v$, then $f_v(e)=i^*$. Set $H_v(\Gamma)=H(f_v) $
and $H(\Gamma)=\otimes_v \, H_v(\Gamma)$ where $v$ runs over all vertices of
$\Gamma$.

Consider now a labeled trivalent graph $\Gamma\subset S^2 $ as above, endowed
with a family of vectors $h=\{h_v \in H_v(\Gamma)\}_v$, where $v$ runs over
all vertices of $\Gamma$. We thicken $\Gamma$ into a $\cat$-colored ribbon
graph on $S^2$ as follows.  First, we insert inside each edge $e$ of $\Gamma$
a coupon with one edge outgoing from the bottom along $e$ and with one edge
outgoing from the top along $e$ in the direction opposite to the one on
$e$. If $e$ is labeled with $i\in I$, then these two new (smaller) edges are
labeled with $V_i$ and $V_{i^*}$, respectively, and the coupon is labeled with
$w_i:V_i\to V_{i^*}^*$ as in Figure \ref{F:Couponw}.

\begin{figure}[h,t]
  \centering $ \xymatrix{ \:
    \\
    *+[F]\txt{ \; $w_i$ \;} \ar[d]^{V_i} \ar[u]_{V_{i^*}}
    \\
    \: \hspace{40pt} } $
 \caption{}\label{F:Couponw}
\end{figure}

Next, we thicken each vertex $v$ of $\Gamma$ to a coupon so that the three
half-edges adjacent to $v$ yield three arrows adjacent to the top side of the
coupon and oriented towards it. If $i,j,k\in I$ are the labels of these arrows
(enumerated from the left to the right), then we color this coupon with the
image of $h_v$ under the natural isomorphism $H_v(\Gamma)\to H^{ijk}$. Denote
the resulting $\cat$-colored ribbon graph by $\Omega_{\Gamma, h}$. Then
$\mathbb G(\Gamma, h)= G (\Omega_{\Gamma, h})$ is an isotopy invariant of the
pair $(\Gamma, h)$ independent of the way in which the vertices of $\Gamma$
are thickened to coupons.

In the next section we describe a related but somewhat different approach to
invariants of colored ribbon graphs and labeled trivalent graphs in $S^2$.

\section{Cutting of graphs and ambidextrous pairs}

Let $T\subset S^2$ be a $\cat$-colored ribbon graph and let $e$ be an edge of
$T$ colored with a simple object $V$ of $\cat$. Cutting $T$ at a point of $e$,
we obtain a $\cat$-colored ribbon graph $T_V $ in $\R\times [0,1]$ with one
input and one output such that the edges incident to the input and output are
oriented downward and colored with $V$.  In other words, $T_V\in
\End_{\Graph_\cat}((V,+))$.  Note that the closure of $T_V$ (obtained by
connecting the endpoints of $T_V$ by an arc in $S^2=\R^2\cup \{\infty\}$
disjoint from the rest of $T_V$) is isotopic to $T$ in~$S^2$.  We call $T_V $
a \emph{cutting presentation} of $T$ and let $\langle T_V\rangle \, \in K$
denote the isotopy invariant of $T_V$ defined from the equality $G(T_V )= \,
\langle T_V \rangle \, \Id_V$.

In the following definition and in the sequel, a ribbon graph is
\emph{trivalent} if all its coupons are adjacent to 3 half-edges.

\begin{definition}\label{def1}   
  Let $\{V_i, w_i \}_{i \in I}$ be basic data in $\cat$.  Let $I_0$ be a
  subset of $I$ invariant under the involution $i\mapsto i^*$ and $\qd:I_0\to
  \FK$ be a mapping such that $\qd(i)=\qd(i^*)$ for all $i\in I_0$. Let
  $\tet_{I_0}$ be the class of $\cat$-colored connected trivalent ribbon
  graphs in $S^2$ such that the colors of all edges belong to the set $\{V_i
  \}_{i \in I}$ and the color of at least one edge belongs to the set $\{V_i
  \}_{i \in I_0}$. The pair $(I_0,\qd)$ is \emph{trivalent-ambidextrous} if
  for any $T\in \tet_{I_0}$ and for any two cutting presentations $T_{V_i},
  T_{V_j}$ of $T$ with $i,j\in I_0$,
  \begin{equation}\label{El} { \qd}(i)\langle T_{V_i} \rangle \, = { \qd}
    (j)\langle
     T_{V_j} \rangle.
  \end{equation}
\end{definition}
To simplify notation, we will say that the pair $(I_0,\qd)$ in Definition
\ref{def1} is \emph{t-ambi}.  For a t-ambi pair $(I_0,\qd)$ we define a
function $G':\tet_{I_0} \rightarrow \FK$ by
\begin{equation}\label{El+}
  G'(T)={  \qd}(i)\langle T_{V_i}\rangle\, ,
\end{equation}
where $T_{V_i} $ is any cutting presentation of $T$ with $i\in I_0$.  The
definition of a t-ambi pair implies that $G'$ is well defined.

The invariant $G'$ can be extended to a bigger class of $\cat$-colored ribbon
graphs in $S^2$. We say that a coupon of a ribbon graph is straight if both
its bottom and top sides are incident to exactly one arrow.  We can remove a
straight coupon and unite the incident arrows into a (longer) edge, see Figure
\ref{F:smooth}. We call this operation {\it straightening}. A {\it
  quasi-trivalent ribbon graph} is a ribbon graph in $S^2$ such that
straightening it at all straight coupons we obtain a trivalent ribbon graph.

\begin{figure}[h,b]
  \centering \hspace{10pt} $\epsh{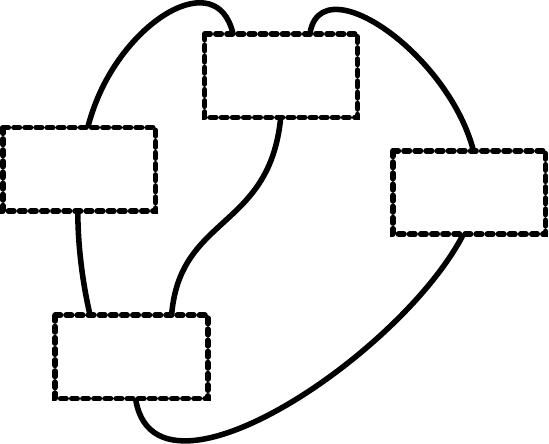}{10ex}$ \hspace{10pt}
  $\longrightarrow$ \hspace{10pt} $\epsh{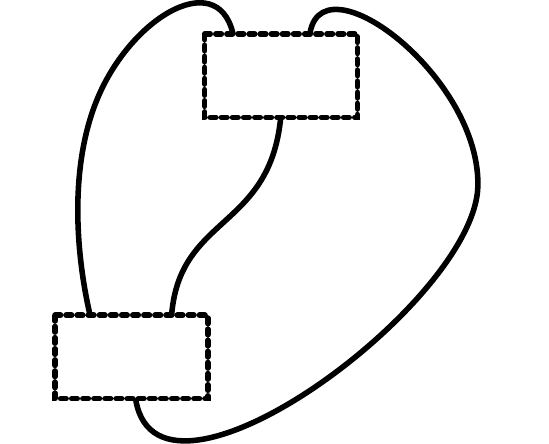}{10ex} $ \hspace{10pt}
  \caption{Straightening a quasi-trivalent  ribbon  graph}\label{F:smooth}
\end{figure}

\begin{lemma}\label{L:generalized}
  Let $(I_0,\qd)$ be a t-ambi pair in $\cat$ and $\overline \tet_{I_0}$ be the
  class of connected quasi-trivalent ribbon graphs in $S^2$ such that the
  colors of all edges belong to the set $\{V_i \}_{i \in I}$, the color of at
  least one edge belongs to the set $\{V_i \}_{i \in I_0}$, and all straight
  coupons are colored with isomorphisms in~$\cat$. Then Formula \eqref{El+}
  determines a well-defined function $G':\overline \tet_{I_0}\rightarrow \FK$.
\end{lemma}
\begin{proof} Given an endomorphism $f$ of a simple object $V$ of $\cat$, we
  write $\langle f \rangle$ for the unique $a\in K$ such that $f=a \Id_V$.
  Observe that if $f:V\to W$ and $g:W\to V$ are homomorphisms of simple
  objects and $f$ is invertible, then
  \begin{equation}\label{El99}
    \langle fg\rangle =\langle gf \rangle\, .
  \end{equation}
  To see this, set $a=\langle fg\rangle\in K $ and $b=\langle gf \rangle\in
  K$. Then
  $$
  a \Id_W= fg(ff^{-1})=f(gf)f^{-1}=b ff^{-1}=b \Id_W\, .
  $$
  Thus, $a=b$.

  Consider now a $\cat$-colored quasi-trivalent ribbon graph $T$ in $S^2$ and
  a straight coupon of $T$ such that the arrow adjacent to the bottom side is
  outgoing and colored with $V_i$, $ i\in I_0 $ while the arrow adjacent to
  the top side is incoming and colored with $V_j$, $ j\in I_0 $.  By
   assumption, the coupon is colored with an isomorphism $ V_i\to V_j$.
  Formula \eqref{El99} implies that $\langle T_{V_i} \rangle = \langle T_{V_j}
  \rangle $, where $T_{V_i}$ is obtained by cutting $T$ at a point of the
  bottom arrow and $T_{V_j}$ is obtained by cutting $T$ at a point of the top
  arrow. By the properties of a t-ambi pair, ${ \qd}(i)={ \qd}(j)$.  Hence
  Formula \eqref{El+} yields the same element of $\FK$ for these two
  cuttings. The cases where some arrows adjacent to the straight coupon are
  oriented up (rather than down) are considered similarly, using the identity
  ${ \qd}(i)={ \qd}(i^*)$. Therefore cutting $T$ at two different edges that
  unite under straightening, we obtain on the right-hand side of \eqref{El+}
  the same element of $\FK$. When we cut $T$ at two edges which do not unite
  under straightening, a similar claim follows from the definitions.
\end{proof}

We can combine the invariant $G'$ with the thickening of trivalent graphs to
obtain invariants of trivalent graphs in $S^2$.  Suppose that $\Gamma\subset
S^2$ is a labeled connected trivalent graph such that the label of at least
one edge of $\Gamma$ belongs to $I_0$. We define
$$
\mathbb{G'}(\Gamma)\in
H(\Gamma)^\star=\Hom_K(H(\Gamma),K)
$$
as follows. Pick any family of vectors $h=\{h_v \in H_v(\Gamma)\}_v$, where
$v$ runs over all vertices of $\Gamma$. The $\cat$-colored ribbon graph
$\Omega_{\Gamma, h}$ constructed at the end of Section \ref{S:InvOfGraphs}
belongs to the class $\overline \tet_{I_0}$ defined in Lemma
\ref{L:generalized}.  Set
$$
\mathbb{G}'(\Gamma)(\otimes_v h_v)=G'(\Omega_{\Gamma, h}) \in \FK.
$$
By the properties of $G'$, the vector $\mathbb{G}'(\Gamma)\in H(\Gamma)^\star$
is an isotopy invariant of~$\Gamma$.  Both $H(\Gamma)$ and
$\mathbb{G'}(\Gamma)$ are preserved under the reversion transformation
inverting the orientation of an edge of $\Gamma$ and replacing the label of
this edge, $i$, with $i^*$. This can be easily deduced from Formula
\eqref{E:FamilyIso}.

\section{Modified $6j$-symbols}\label{S:Modified6j}

Let $\cat$ be a pivotal tensor Ab-category with ground ring $K$, basic data
$\{V_i, w_i:V_i \to V_{i^*}^*\}_{i \in I}$, and t-ambi pair $(I_0,\qd)$. To
simplify the exposition, we assume that there is a well-defined direct
summation of objects in $\cat$.  We define a system of tensors called modified
$6j$-symbols.

Let $i,j,k,l,m,n$ be six elements of $ I$ such that at least one of them is in
$I_0$.  Consider the labeled trivalent graph
$\Gamma=\Gamma(i,j,k,l,m,n)\subset\R^2\subset S^2$ given in Figure
\ref{F:Gamma7777}. By definition,
$$
H(\Gamma)= H(i,j,k^*) \otimes_K H(k,l,m^*) \otimes_K H(n,l^*, j^*)\otimes_K
H(m,n^*, i^*).
$$
\begin{figure}[h, b]
  \centering \hspace{10pt} $ \epsh{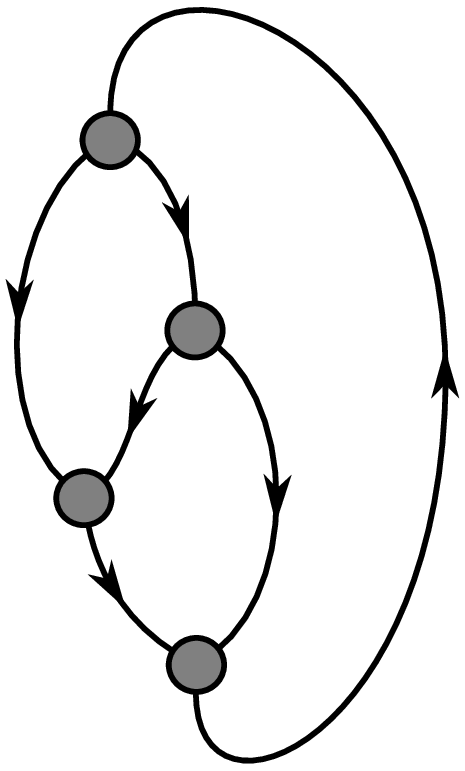}{120pt}\pic{.8pt}
  {\put(-83,-49){k}\put(-51,26){n}\put(-98,10){i}
    \put(-68,-3){j}\put(-46,-39){l}\put(-18,-7){m}} $
  \caption{$\Gamma(i,j,k,l,m,n)$}\label{F:Gamma7777}
\end{figure}

We define the modified $6j$-symbol of the tuple $(i,j,k,l,m,n)$ to be
\begin{equation}\label{mss}
  \sjtop ijklmn=\mathbb{G}'(\Gamma)\in H(\Gamma)^\star=\Hom_K(H(\Gamma), K).
\end{equation}
It follows from the definitions that the modified $6j$-symbols have the
symmetries of an oriented tetrahedron.  In particular,
$$
\sjtop ijklmn = \sjtop j{k^*}{i^*}mnl = \sjtop klm{n^*}{i}{j^*}.
$$
These equalities hold because the labeled trivalent graphs in $S^2$ defining
these $6j$-symbols are related by isotopies and reversion transformations
described above.

To describe $H(\Gamma)^\star$, we compute the duals of the symmetrized
multiplicity modules.  For any indices $i,j,k\in I$ such that at least one of
them lies in $I_0$, we define a pairing
\begin{equation} \label{E:pairing1-} (,)_{ijk}: H(i,j,k)\otimes_K H(k^*, j^*,
  i^*)\to K
\end{equation}
by
$$
  (x,y)_{ijk}=\mathbb G'(\Theta) (x\otimes y)\, ,
$$
where $x\in H(i,j,k)=H^{ijk}$ and $y\in H(k^*, j^*, i^*)=H^{k^* j^* i^*}$ and
$\Theta=\Theta_{i,j,k} $ is the theta graph with vertices $u,v$ and three
edges oriented from $v$ to $u$ and labeled with $i,j,k$, see\ Figure
\ref{F:Thetaijk}. Clearly,
$$H_u(\Theta)=H^{ijk}=H(i,j,k)\,\, \,\, {\text {and}} \,\,\,\,
H_v(\Theta)=H^{k^*j^*i^*}=H(k^*,j^*, i^*)$$ so that we can use $x$, $y$ as the
colors of $u$, $v$, respectively.  It follows from the definitions that the
pairing $(,)_{ijk}$ is invariant under cyclic permutations of $i,j,k$ and
$(x,y)_{ijk}=(y,x)_{k^*j^*i^*}$ for all $x\in H(i,j,k) $ and $y\in H(k^*, j^*,
i^*)$.

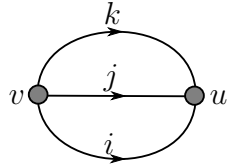
\begin{figure}[h, b]
 \centering \hspace{10pt} $\epsh{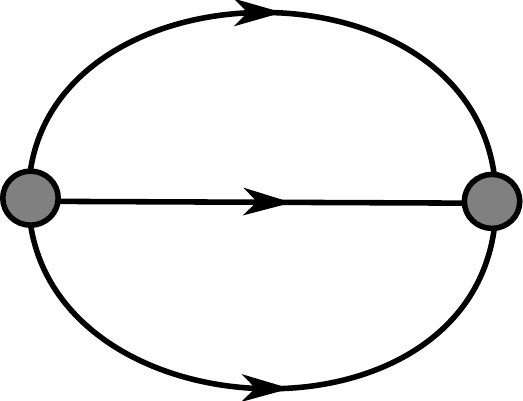}{10ex} $
  \put(-38,-19){$i$}\put(-38,7){$j$}\put(-38,30){$k$}
  \put(-74,0){$v$}\put(2,0){$u$}
  \caption{$\Theta_{i,j,k}$}\label{F:Thetaijk}
\end{figure}

We now give sufficient conditions for the pairing $(,)_{ijk}$ to be
non-degenerate.  We say that a pair $(i,j)\in I^2$ is \emph{good} if
$V_i\otimes V_j$ splits as a finite direct sum of some $V_k$'s (possibly with
multiplicities) such that $k\in I_0$ and $\qd(k) $ is an invertible element of
$K$. By duality, a pair $(i,j)\in I^2$ is good if and only if the pair $(j^*,
i^*)$ is good. We say that a triple $(i,j,k)\in I^3$ is \emph{good} if at
least one of the pairs $(i,j), (j,k), (k, i)$ is good. Clearly, the goodness
of a triple $(i,j,k) $ is preserved under cyclic permutations and implies the
goodness of the triple $( k^*, j^*, i^*) $. Note also that if a triple
$(i,j,k) $ is good, then at least one of the indices $i,j,k$ belongs to $I_0$
so that we can consider the pairing (\ref{E:pairing1-}).

\begin{lemma}\label{L:nondegeneracy} If the triple $(i,j,k)\in I^3$ is good,
  then the pairing \eqref{E:pairing1-} is non-degenerate.
\end{lemma}

\begin{proof} For any $i,j,k\in I$, consider the $K$-modules
  \begin{align*}
    \label{}
    H^{ij}_k = \Hom(V_k,V_i \otimes V_j) \quad {\text{and}} \quad H_{ij}^k &=
    \Hom(V_i \otimes V_j,V_k).
  \end{align*}
  Consider the homomorphism
  \begin{equation}\label{E:Hiso}
    H^{ij}_k \rightarrow
    H^{ijk^*}, \,\, y \mapsto (y \otimes w_{k^*}^{-1})\, b_{V_k}.
  \end{equation}
  It is easy to see that this is an isomorphism.  Composing with the canonical
  isomorphism $H^{ijk^*} \rightarrow H(i,j,k^*)$ we obtain an isomorphism
  $a_k^{ij}:H_k^{ij}\rightarrow H(i,j,k^*)$.  Similarly, we have an
  isomorphism
  $$
  a^k_{ij}:H_{ij}^k \rightarrow H^{kj^*i^*}=H(k,j^*,i^*), \,\, x \mapsto
  (x\otimes w_{j^*}^{-1} \otimes w_{i^*}^{-1})(\Id_{V_i}\otimes b_{V_j}
  \otimes \Id_{V^*_i})b_{V_i}.
  $$

  Let $(,)^{ij}_k:H^{ij}_k \otimes_K H^k_{ij} \rightarrow \FK$ be the bilinear
  pairing whose value on any pair $(y\in H^{ij}_k, x\in H^k_{ij})$ is computed
  from the equality $$(y,x)^{ij}_k \Id_{V_k}= \qd(k) xy:V_k\to V_k.$$ It
  follows from the definitions that under the isomorphism
  $$
  a_{k}^{ij}\otimes a^k_{ij}:H^{ij}_k \otimes_K H^k_{ij}\to
  H(i,j,k^*)\otimes_K H(k,j^*,i^*)
  $$
  the pairing $(,)^{ij}_k$ is transformed into $(,)_{ijk^*}$.

  We can now prove the claim of the lemma.  By cyclic symmetry, it is enough
  to consider the case where the pair $(i,j)$ is good. Then $H^{ij}_k $ and $
  H^k_{ij}$ are free $K$-modules of the same finite rank and the pairing
  $(,)^{ij}_k$ is non-degenerate. Therefore the pairing $(,)_{ijk^*}$ is
  non-degenerate.
\end{proof}

We say that a 6-tuple $(i,j,k,l,m,n)\in I^6$ is \emph{good} if all the triples
\begin{equation}\label{tritri}
(i,j,k^*), (k,l,m^*), (n,l^*, j^*), (m,n^*, i^*)
\end{equation}
are good.  This property of a 6-tuple is invariant under symmetries of an
oriented tetrahedron acting on the labels of the edges.  The previous lemma
implies that if $(i,j,k,l,m,n)$ is good, then the ambient module of the
modified $6j$-symbol \eqref{mss} is
\begin{equation}\label{momo}
 H(k,j^*,i^*) \otimes_K
H(m,l^*, k^*) \otimes_K H(j,l, n^* )\otimes_K H(i,n,m^*).
\end{equation}

\section{Comparison with the standard $6j$-symbols}\label{S:Prop}

We compare the modified $6j$-symbols defined above with standard $6j$-symbols
derived from decompositions of tensor products of simple objects into direct
sums.

We keep notation of Section \ref{S:Modified6j} and  begin with two
simple lemmas.

\begin{lemma}
  For any $i,j, k,l,m\in I$, the formula $(f,g)\mapsto (f\otimes \Id_{V_l})g$
  defines a $K$-homomorphism
\begin{equation}
  \label{E:homo1}
  H^{ij}_k \otimes_K H_m^{kl}  \rightarrow \Hom(V_m,V_i\otimes V_j\otimes V_l).
\end{equation}
If $(i,j)$ is good, then the direct sum of these homomorphisms is an
isomorphism
\begin{equation}
  \label{E:iso1}
  \bigoplus_{k\in I}H^{ij}_k \otimes_K H_m^{kl}  \rightarrow \Hom(V_m,V_i\otimes
  V_j\otimes V_l).
\end{equation}
\end{lemma}
\begin{proof} 
  The first claim is obvious; we prove the second claim.  If $(i,j)$ is good,
  then for every $k\in I$, the $K$-modules $H^{ij}_k$ and $H_{ij}^k$ are free
  of the same finite rank. Fix a basis $\{\alpha_{k,r}\}_{r\in R_k}$ of
  $H^{ij}_k$ where $R_k$ is a finite indexing set.  Then there is a basis
  $\{\alpha^{k,r}\}_{r\in R_k}$ of $H_{ij}^k$ such that $
  \alpha^{k,r}\alpha_{k,s}=\delta_{r,s}\Id_{V_k}$ for all $r,s \in R_k$.
  Clearly,
  $$
  \Id_{V_i\otimes V_j}=\sum_{k\in I, r\in R_k} \alpha_{k,r}\alpha^{k,r}.
  $$
  We define a $K$-homomorphism $\Hom(V_m,V_i\otimes V_j\otimes V_l)\rightarrow
  \bigoplus_{k\in I} H^{ij}_k \otimes H_m^{kl} $ by
  $$
  f\mapsto \sum_{k\in I,\, r\in R_k}\alpha_{k,r}\otimes_K \, ( \alpha^{k,r}
  \otimes \Id_{V_l})f.
  $$
  This is the inverse of the homomorphism~\eqref{E:iso1}.
\end{proof}
Similarly, we have the following lemma.
\begin{lemma}\label{BH}
  For any $i,j,l,m,n \in I$, the formula $(f,g)\mapsto (\Id_{V_i}\otimes f)g$
  defines a $K$-homomorphism
\begin{equation}
\label{E:homo2}
H^{jl}_n \otimes_K H_m^{in}  \rightarrow \Hom(V_m,V_i\otimes V_j\otimes V_l).
\end{equation}
If the pair $(j,l)$ is good, then the direct sum of these homomorphisms is an
isomorphism
\begin{equation}
  \label{E:iso2}
  \bigoplus_{n\in I} H^{jl}_n \otimes_K H_m^{in} \rightarrow
  \Hom(V_m,V_i\otimes V_j\otimes V_l).
\end{equation}
\end{lemma}

Suppose that both pairs $(i,j)$ and $(j,l)$ are good.  Composing the
isomorphism \eqref{E:iso1} with the inverse of the isomorphism
\eqref{E:iso2} we obtain an isomorphism
\begin{equation}
  \label{E:iso4}
  \bigoplus_{k\in I}H^{ij}_k \otimes_K H_m^{kl}  \rightarrow \bigoplus_{n\in
    I} H^{jl}_n \otimes_K H_m^{in}.
\end{equation}
Restricting to the summand on the left-hand side corresponding to $k$ and
projecting to the summand on the right-hand side corresponding to $n$ we
obtain a homomorphism
\begin{equation}
  \label{E:iso3}
  \sj ijklmn : H_k^{ij} \otimes_K H_m^{kl} \rightarrow H^{jl}_n \otimes_K
  H_m^{in}.
\end{equation}
This is the (standard) $6j$-symbol determined by $ i,j,k,l,m,n$. We emphasize
that it is defined only when the pairs $(i,j)$ and $(j,l)$ are good.  Note
that this $6j$-symbol is equal to zero unless both $k$ and $n$ belong to $
I_0$.

The following equality is an equivalent graphical form of the same
definition. It indicates that the composition of the homomorphisms
\eqref{E:iso3} and \eqref{E:homo2} summed up over all $n\in I$ is equal to the
homomorphism \eqref{E:homo1}:
$$
\pic{1ex}{
  \begin{picture}(6,11)(0,0)
    \multiput(2,2)(4,0){2}{\line(0,1){2}}
    \multiput(2,2)(0,2){2}{\line(1,0){4}}
    \multiput(0,7)(4,0){2}{\line(0,1){2}}
    \multiput(0,7)(0,2){2}{\line(1,0){4}}
    \put(1,11){\vector(0,-1){2}}
    \put(3,11){\vector(0,-1){2}}
    \put(5,11){\vector(0,-1){7}}
    \put(2,7){\vector(1,-3){1}}
    \put(4,2){\vector(0,-1){2}}
    \put(1.2,10){$\mathsmall i$}
    \put(3.2,10){$\mathsmall j$}
    \put(5.2,10){$\mathsmall l$}
    \put(1.4,5){$\mathsmall k$}
    \put(4.2,1){$\mathsmall m$}
  \end{picture}} \,=\, \sum_{n\in I}\,
\pic{1ex}{
  \begin{picture}(6,11)(0,0)
    \multiput(0,2)(4,0){2}{\line(0,1){2}}
    \multiput(0,2)(0,2){2}{\line(1,0){4}}
    \multiput(2,7)(4,0){2}{\line(0,1){2}}
    \multiput(2,7)(0,2){2}{\line(1,0){4}}
    \put(5,11){\vector(0,-1){2}}
    \put(3,11){\vector(0,-1){2}}
    \put(1,11){\vector(0,-1){7}}
    \put(3.9,7){\vector(-1,-3){1}}
    \put(2,2){\vector(0,-1){2}}
    \put(1.2,10){$\mathsmall i$}
    \put(3.2,10){$\mathsmall j$}
    \put(5.2,10){$\mathsmall l$}
    \put(4,5){$\mathsmall n$}
    \put(2.2,1){$\mathsmall m$}
  \end{picture}}\circ \sj ijklmn\, .
$$

Our next aim is to relate the $6j$-symbol \eqref{E:iso3} to the modified
$6j$-symbol \eqref{mss}.  Recall the isomorphism
$a_{k}^{ij}:H_k^{ij}\rightarrow H(i,j,k^*)$ introduced in the proof of Lemma
\ref{L:nondegeneracy}. Using these isomorphisms, we can rewrite the
$6j$-symbol \eqref{E:iso3} as a homomorphism
\begin{equation}\label{hl}
  H(i,j,k^*) \otimes_K H(k,l,m^*) \rightarrow H(j,l,n^*) \otimes_K
  H(i,n,m^*).
\end{equation}
Since $(j,l)$ is good, $H( j,l, n^*)^\star=H(n,l^*,j^*)$. Assuming that
$(i,n)$ is good, we write $H(i,n,m^*)^\star=H(m,n^*, i^*)$ and consider the
homomorphism
$$
\sjs ijklmn: H(i,j,k^*) \otimes_K H(k,l,m^*) \otimes_K  H(n,l^*,j^*)
\otimes_K H(m,n^*, i^*) \to K
$$
adjoint to (\ref{hl}). This homomorphism has the same source module as the
modified $6j$-symbol \eqref{mss}.

\begin{lemma}
  \label{L:two6j}
  For any $i,j,k,l,m,n\in I$ such that the pairs $(i,j), (j,l)$, $(i,n)$ are
  good and $m,n\in I_0$,
  \begin{equation}
  \label{E:sjtopsym}
  \sjs ijklmn= \qd(n) \sjtop ijklmn
  \, .
  \end{equation}
\end{lemma}
\begin{proof}
  Pick any $ x_1 \in H^{ij}_k$, $x_2\in H_m^{kl}$.  By Lemma \ref{BH},
  $$
  (x_1\otimes \Id_{V_l})x_2=\sum_{n'\in I} \sum_{r\in R_{n'} }
  (\Id_{V_i}\otimes y_1^{n',r})y_2^{n',r} ,
  $$
  where $y_1^{n',r}\in H^{jl}_{n'}$, $y_2^{n',r}\in H^{in'}_m$, and $R_{n'}$
  is a finite set of indices.  By definition,
  $$
  \sj ijklmn (x_1 \otimes x_2)=\sum_{r\in R_{n} } y_1^{n,r}\otimes y_2^{n,r}.
  $$
  Under our assumptions on $i,j,k,l,m,n$, the pairings $(\, ,)^{jl}_n$ and
  $(\, ,)^{in}_m$ introduced in the proof of Lemma \ref{L:nondegeneracy} are
  non-degenerate, and we use them to identify $H^{jl}_n=(H^n_{jl})^\star $ and
  $H^{in}_m=(H^m_{in})^\star $.  Consider the homomorphism
\begin{equation}
  \label{E:iso367}
  {\sj ijklmn}_+ : H_k^{ij} \otimes_K H_m^{kl} \otimes_K H_{jl}^n \otimes_K
  H^m_{in} \rightarrow K
\end{equation}
adjoint to $\sj ijklmn$. The computations above show that
\begin{equation}
  \label{E:leftsj} {\sj ijklmn}_+ (x_1\otimes x_2 \otimes x_3 \otimes x_4)=
  \sum_{r\in R_n}\,\, (y_1^{n,r}, x_3)^{jl}_n \,\, (y_2^{n,r},x_4)^{in}_m
\end{equation}
for any $x_3\in H_{jl}^n$ and $x_4\in H_{in}^m$.

Consider the $\cat$-colored ribbon graphs $\Gamma_1$, $\Gamma^{n',r}_2$,
$\Gamma^{n,r}_3$ in Figure \ref{F:gammas} (it is understood that an edge with
label $s\in I$ is colored with $V_s$).  It is clear that
\begin{equation}
\label{E:rightsj}
G'\left(\Gamma_1\right)=   \sum_{n'\in I, r\in R_{n'}}
G'\left(\Gamma^{n',r}_2\right)  =   \sum_{n'\in I, r\in R_{n'}} \delta_{n,n'}
\qd(n)^{-1} (y_1^{n,r},x_3)^{jl}_n \,
G'\left(\Gamma^{n,r}_3\right)
\end{equation}
where the second equality follows from the fact that $xy=\qd(n)^{-1}
(y,x)\Id_{V_n}$ for $x\in H^n_{jl}$ and $y \in H^{jl}_n $.  Similarly,
$G'\left(\Gamma^{n,r}_3\right)= (y_2^{n,r},x_4)^{in}_m$ (here we use the
inclusion $m\in I_0$). Therefore
$$
\qd(n)\, G'\left(\Gamma_1\right) = {\sj ijklmn}_+ (x_1\otimes x_2 \otimes x_3
\otimes x_4).
$$
Rewriting this equality in terms of the symmetrized multiplicity modules, we
obtain the claim of the lemma.
\end{proof}
\begin{figure}[h, b] \centering 
  \subfloat[$\Gamma_1$]{\label{F:gamma1} \hspace{10pt} $\epsh{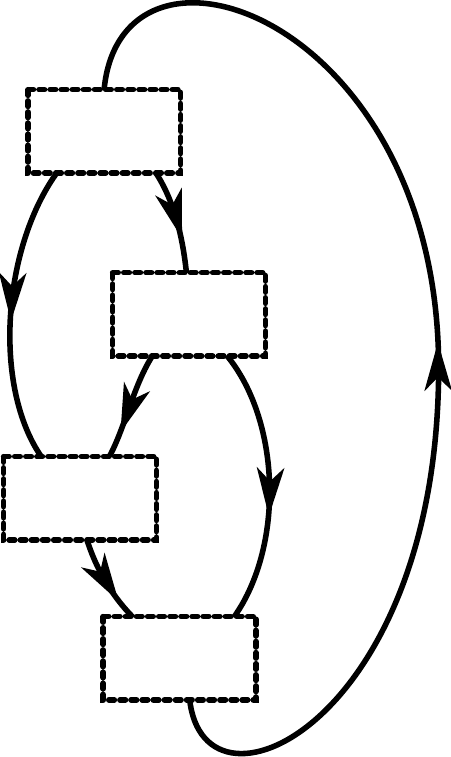}{24ex}$
    \put(-68,-40){$k$}\put(-41,32){$n$}\put(-80,16){$i$}
    \put(-62,-4){$j$}\put(-36,-19){$l$}\put(-14,3){$m$}
    \put(-61,41){$x_4$}\put(-47,11){$x_3$}
    \put(-67,-19){$x_1$}\put(-51,-45){$x_2$}
    \hspace{10pt} } \subfloat[$\Gamma^{n',r}_2$]{\label{F:gamma2}
    \hspace{10pt} $\epsh{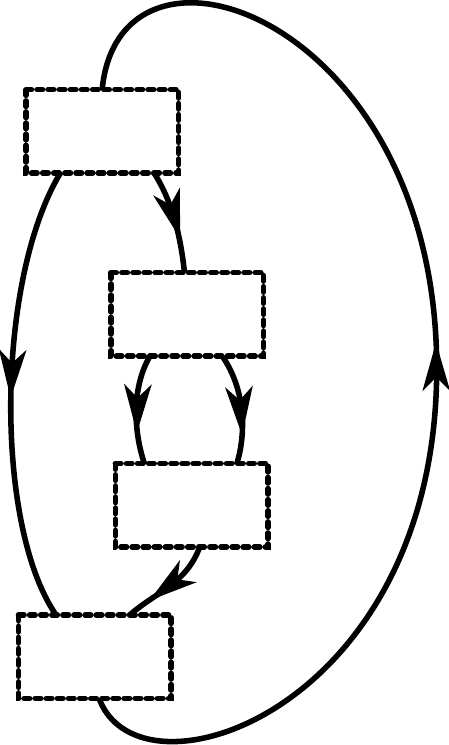}{24ex}$
    \put(-41,-36){$n'$}\put(-41,32){$n$}\put(-81,0){$i$}
    \put(-62,-9){$j$}\put(-32,-9){$l$}\put(-17,0){$m$}
    \put(-61,41){$x_4$}\put(-47,11){$x_3$} \put(-50,-22){{\scriptsize
        $y_1^{n',r}$}}\put(-66,-47){{\scriptsize $y_2^{n',r}$}}
    \hspace{10pt} } \subfloat[$\Gamma^{n,r}_3$]{\label{F:gamma3}
    \hspace{10pt} $\epsh{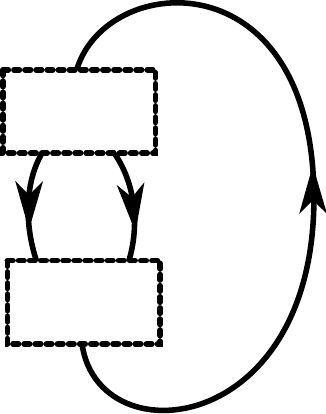}{15ex}$
    \put(-53,1){$i$}\put(-33,1){$n$}\put(-15,1){$m$}
    \put(-51,17){$x_4$}\put(-52,-18){{\scriptsize $y_2^{n,r}$}}
    \hspace{10pt}
  }
  \caption{}
  \label{F:gammas}
\end{figure}

We say that a 6-tuple $(i,j,k,l,m,n)\in I^6$ is \emph{strongly good} if
$m,n\in I_0$ and the pairs $(i,j)$, $ (j,l)$, $(i,n)$, and $(k,l)$ are good. A
strongly good 6-tuple $(i,j,k,l,m,n)$ is good in the sense of Section
\ref{S:Modified6j}, so that both associated $6j$-symbols $\sjs ijklmn$ and
$\sjtop ijklmn$ lie in the $K$-module (\ref{momo}). Lemma \ref{L:two6j} yields
equality \eqref{E:sjtopsym} understood as an equality in this $K$-module.

Unfortunately, the notion of a strongly good 6-tuple is not invariant under
symmetries of an oriented tetrahedron. To make it symmetric, one has to add
more conditions on the labels. We say that a tuple $(i,j,k,l,m,n)\in I^6$ is
\emph{admissible} if all the indices $i,j,k,l,m,n$ belong to $I_0$ and the
pairs $(i,j)$, $ (j,l)$, $(i,n)$, $ (k,l)$, $(j,k^*)$, $(k^*,i)$, $(l,m^*)$,
$(m^*,k)$, $(n,l^*)$, $(j^*,n)$, $(m,n^*)$, $(i^*,m)$ are good. Admissible
6-tuples are strongly good, and the notion of an admissible 6-tuple is
preserved under the symmetries of an oriented tetrahedron.

\section{Properties of the modified $6j$-symbols}\label{S:Prop+}

Given a good triple $(i,j,k)\in I^3$ and a tensor product of several
$K$-modules such that among the factors there is a matched pair $H(i,j,k)$, $
H(k^*, j^*, i^*)$, we may contract any element of this tensor product using
the pairing \eqref{E:pairing1-}. This operation is called the contraction
along $H(i,j,k)$ and denoted by $*_{ijk}$.  For example, an element $x\otimes
y \otimes z \in H(i,j,k)\otimes_K H(k^*, j^*, i^*) \otimes_K H$, where $H$ is
a $K$-module, contracts into $(x,y)z\in H$.

\begin{theorem}[The Biedenharn-Elliott identity]
  \label{T:BEId} Let $j_0,j_1,{\ldots},j_8 $ be elements of $ I $ such that
  the tuples $(j_1, j_2, j_5, j_8, j_0, j_7)$ and $(j_5, j_3, j_6, j_4, j_0,
  j_8)$ are strongly good, $ j_7, j_8 \in I_0$, and the pair $(j_2,j_3) $ is
  good. Set $$J=\{ j \in I\, \vert\, H_j^{j_2 j_3}\neq 0 \} \subset I_0.$$ If
  the pairs $(j_1, j)$ and $(j,j_4)$ are good for all $j\in J$, then all
  6-tuples defining the $6j$-symbols in the following formula are strongly
  good and
  $$
  \sum_{j\in J} \qd(j) \, *_{j_2j_3j^*} *_{jj_4j_7^*}*_{j_1j j_6^*}
  \left(\sjtop {j_1}{j_2}{j_5}{j_3}{j_6}{j}\otimes \sjtop
    {j_1}{j}{j_6}{j_4}{j_0}{j_7}\otimes \sjtop {j_2}{j_3}{j}{j_4}{j_7}{j_8}
  \right)
  $$
  \begin{equation}\label{E:BEId}
    = *_{j_5j_8j_0^*}\left(\sjtop {j_5}{j_3}{j_6}{j_4}{j_0}{j_8}\otimes
      \sjtop {j_1}{j_2}{j_5}{j_8}{j_0}{j_7}\right)\, .
  \end{equation}
  Here both sides lie in the tensor product of six $K$-modules
  $$
  H(j_6,j_3^*,j_5^*), H(j_5,j_2^*,j_1^*), H(j_0,j_4^*,j_6^*),
  H(j_1,j_7,j_0^*), H(j_2,j_8,j_7^*), H(j_3,j_4,j_8^*).
  $$
\end{theorem}
\begin{proof} The claim concerning the strong goodness follows directly from
  the definitions. We verify \eqref{E:BEId}. Recall that if $(i,j)$ is a good
  pair, and $V_k$ is a summand of $V_i\otimes V_j$, then $\qd (k)$ is
  invertible. Therefore if $\qd(j_7)$ is not invertible, then $V_{j_7}$ can
  not be a summand of $V_{j_2}\otimes V_{j_8}$, and so both sides of
  \eqref{E:BEId} are equal to 0.  Similarly if $\qd(j_8)$ is not invertible,
  then $V_{j_8}$ can not be a summand of $V_{j_3}\otimes V_{j_4}$, and so both
  sides of \eqref{E:BEId} are equal to 0. We assume from now on that
  $\qd(j_7)$ and $\qd(j_8)$ are invertible in $K$.

  The Pentagon Axiom for the tensor multiplication in $\cat$ implies that
  \begin{align}
    \label{E:usualBE}
    \sum_{j\in J}&\left(I_{j_0}^{j_1 j_7} \otimes \sj
      {j_2}{j_3}{j}{j_4}{j_7}{j_8}\right)\left( \sj
      {j_1}{j}{j_6}{j_4}{j_0}{j_7}\otimes I_{j_2 j_3}^{j}\right)\left(I_{j_6
        j_4}^{j_0}  \otimes \sj {j_1}{j_2}{j_5}{j_3}{j_6}{j} \right)\notag\\
    &= \left(\sj {j_1}{j_2}{j_5}{j_8}{j_0}{j_7} \otimes I_{j_8}^{j_3
        j_4}\right)P_{23}\left(\sj {j_5}{j_3}{j_6}{j_4}{j_0}{j_8}\otimes
      I_{j_5}^{j_1 j_2}\right)
  \end{align}
  where $I_i^{jk}$ is the identity automorphism of $H_i^{jk}$, $I^i_{jk}$ is
  the identity automorphism of $H^i_{jk}$, and $P_{23}$ is the permutation of
  the second and third factors in the tensor product (cf.\ Theorem VI.1.5.1 of
  \cite{Tu}).  Both sides of \eqref{E:usualBE} are homomorphisms
  $$
  H^{j_6j_4}_{j_0}\otimes H^{j_5j_3}_{j_6}\otimes H^{j_1j_2}_{j_5}\rightarrow
  H^{j_1j_7}_{j_0}\otimes H^{j_2j_8}_{j_7}\otimes H^{j_3j_4}_{j_8}\, .
  $$

  We can rewrite all $6j$-symbols in \eqref{E:usualBE} in the form \eqref{hl}.
  Note that the homomorphism \eqref{hl} carries any $x\in H(i,j,k^*) \otimes_K
  H(k,l,m^*)$ to
  $$
  *_{ijk^*} *_{klm^*}\left(\sjs ijklmn \otimes x \right)= \qd(n) *_{ijk^*}
  *_{klm^*}\left(\sjtop ijklmn \otimes x \right)\, ,
  $$
  where we suppose that the tuple $(i,j,k,l,m,n)$ is strongly good and
  consider both $6j$-symbols as vectors in the module \eqref{momo}.

  Denote by $C_j$ the tensor product of the three $6j$-symbols in the $j$-th
  term of the sum on the left hand side of \eqref{E:BEId}.  Denote by $D$ the
  tensor products of the two $6j$-symbols on the right hand side of
  \eqref{E:BEId}.  We can rewrite \eqref{E:usualBE} as
  \begin{align*}\sum_{j\in J}\qd({j_8})\qd({j_7}) \qd(j)
    *_{j_2j_3j^*} *_{jj_4j_7^*}*_{j_1j j_6^*} *_{j_6j_4j_0^*}
    *_{j_1j_2j_5^*}*_{j_5j_3j_6^*} (C_j\otimes x)\\
    =\qd({j_8})\qd({j_7}) *_{j_1j_2j_5^*} *_{j_5j_8j_0^*} *_{j_5j_3j_6^*}
    *_{j_6j_4j_0^*} (D\otimes x)
  \end{align*}
  for all $x\in H({j_6,j_4,j_0^*})\otimes H({j_5,j_3,j_6^*})\otimes
  H({j_1,j_2,j_5^*})$.  Since $\qd({j_7})$ and $\qd({j_8})$ are invertible
  elements of $K$, the previous equality implies that
  \begin{equation*}
    \label{E:sBEI}
    \sum_{j\in J} \qd(j)*_{j_2j_3j^*} *_{jj_4j_7^*}*_{j_1j j_6^*}(C_j)=
    *_{j_5j_8j_0^*}(D)\, .
 \end{equation*}
 This proves the theorem.
\end{proof}

\begin{theorem}[The orthonormality relation]
  \label{T:orth}
  Let $i,j,k,l,m,p$ be elements of $I$ such that $k, m \in I_0$ and the pairs
  $(i,j)$, $(j,l)$, $(p,l)$, $(k,l)$ are good.  Set
  $$
  N= \{n\in I\, |\, H^{jl}_n\neq 0\,\, {\text {and}} \,\, H^{in}_m\neq 0\}
  \subset I_0 \, .
  $$
  If the pair $(i,n)$ is good for all $n\in N$, then both 6-tuples defining
  the $6j$-symbols in the following formula are good (in fact, the first one
  is strongly good) and
  $$
  \qd(k)\sum_{n\in N}\qd(n) *_{inm^*}*_{jln^*}\left(\sjtop ij{p}lmn
    \otimes \sjtop k{j^*}inml \right)
  $$
  $$
  =\delta_{k,p}\Id(i,j,k^*) \otimes \Id(k,l,m^*),
  $$
  where $\delta_{k,p}$ is the Kronecker symbol and $\Id(a,b,c)$ is the
  canonical element of $H(a,b,c)\otimes_K H(c^*,b^*,a^*)$ determined by the
  duality pairing.
\end{theorem}
\begin{proof} Consider any $i,j, l,m \in I $ such that the pairs $(i,j)$ and
  $(j,l)$ are good and consider the associated isomorphism \eqref{E:iso4}.
  Restricting the inverse isomorphism to the summand in the source
  corresponding to $n\in I$ and projecting into the summand in the target
  corresponding to $k\in I$ we obtain a homomorphism
\begin{equation}
\label{E:sjinv}
\sj ijklmn_{\text{inv}}:   H^{jl}_n \otimes_K H_m^{in}\rightarrow H_k^{ij}
\otimes_K H_m^{kl}.
\end{equation}
These homomorphisms corresponding to fixed $i,j,l,m$ and various $k,n\in I$
form a block-matrix of the isomorphism inverse to \eqref{E:iso4}.  Therefore,
\begin{equation}
\label{E:sjcompinv}
\sum_{n\in I}\, \sj ijklmn_{\text{inv}}\circ \sj ij{p}lmn =\delta_{p}^k( i,j,l,m)
\end{equation}
where $\delta_{p}^k( i,j,l,m) $ is zero if $k\neq p$ and is the identity
automorphism of $H_k^{ij} \otimes_K H_m^{kl}$ if $k=p$.

Switching to the symmetrized multiplicity modules as in Section \ref{S:Prop},
we can rewrite~\eqref{E:sjinv} as a homomorphism
\begin{equation}
\label{E:sjtopinv}
  H(j,l,n^*) \otimes_K
H(i,n,m^*) \rightarrow H(i,j,k^*) \otimes_K H(k,l,m^*).
\end{equation}
Since $(i,j)$ is good, $H(i,j,k^*)^\star=H(k,j^*, i^*)$. Assuming that $(k,l)$
is good, we can write $H(k,l,m^*)^\star=H(m,l^*, k^*)$. Consider the
homomorphism
$$
\sjs ijklmn_{\text{inv}}: H(j,l,n^*) \otimes_K
H(i,n,m^*) \otimes_K H(k,j^*, i^*) \otimes_K H(m,l^*, k^*) \to K
$$
adjoint to \eqref{E:sjtopinv}. This homomorphism has the same source
module as the $6j$-symbol $\sjtop k{j^*}inml$, where we suppose that
$k\in I_0 $.

\begin{figure}[b]
  \centering
  \subfloat[$\Gamma_4$]{\label{F:gamma4} \hspace{10pt}
    $\epsh{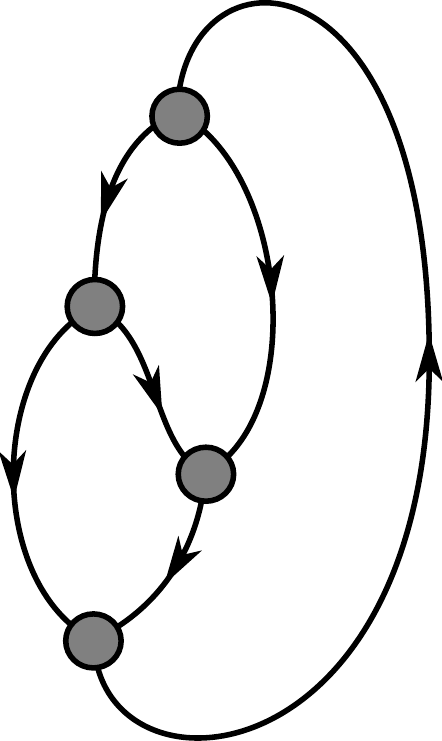}{120pt}\pic{.8pt}{\put(-82,-22){$i$}\put(-54,-2){$j$}
      \put(-31,20){$l$}\put(-79,35){$k$}\put(-51,-44){$n$}\put(-17,2){$m$}}$
    \hspace{10pt} }
  \subfloat[$\Gamma_5$]{\label{F:gamma5} \hspace{10pt}
    $\epsh{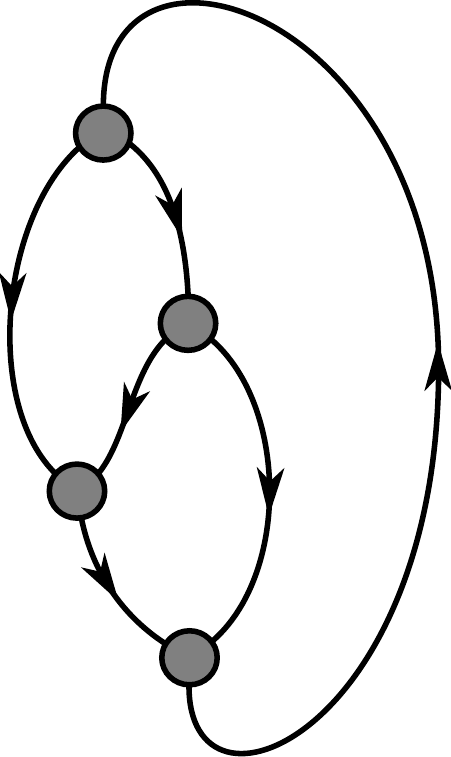}{120pt}\pic{.8pt}{\put(-83,-49){$i$}\put(-51,26){$l$}
      \put(-98,10){$k$}
      \put(-71,-5){$j^*$}\put(-46,-39){$n$}\put(-18,-7){$m$}}$
    \hspace{10pt}
  }
  \caption{}
  \label{F:gammas2}
\end{figure}
The argument similar to that in Lemma \ref{L:two6j} (using the
graphs $ \Gamma_4 $ and $\Gamma_5$ in Figure \ref{F:gammas2}) shows
that if the pairs $(i,j), (j,l), (k,l)$ are good and $k, m\in I_0$,
then
\begin{equation}
\label{E:2sj2}
\sjs ijklmn_{\text{inv}}= \qd(k) \sjtop k{j^*}inml  \, .
\end{equation}

Assuming additionally that the pair $(i, n)$ is good, we can view both
$6j$-symbols in \eqref{E:2sj2} as vectors in
$$
H(n,l^*,j^*) \otimes_K
H(m,n^*,i^*) \otimes_K H(i,j , k^*) \otimes_K H(k,l , m^*)\, .
$$
Then the homomorphism \eqref{E:sjtopinv} carries any $y\in H(j,l,n^*) \otimes_K
H(i,n,m^*)$ to
$$
*_{jln^*}*_{inm^*} \left(\sjs ijklmn_{\text{inv}} \otimes y \right)=
\qd(k) *_{jln^*}*_{inm^*}\left(\sjtop k{j^*}inml \otimes y\right)  \, .
$$
Now we can rewrite \eqref{E:sjcompinv} as
$$
\sum_{n\in N}\qd(n)\qd(k)
*_{jln^*}*_{inm^*}*_{ijk^*}*_{klm^*}\left(\sjtop ij{p}lmn \otimes \sjtop
  k{j^*}inml \otimes x\right) =\delta_{k,p} \, x
$$
for all $x\in H(i,j,k^*)\otimes_K H(k,l,m^*)$.  This implies the
claim of the theorem.
\end{proof}


\section{Ambidextrous objects and standard $6j$-symbols}

To apply the results of the previous sections, we must construct a
pivotal Ab-category $\cat$ with basic data $\{V_i, w_i \}_{i \in I}$
and a t-ambi pair $(I_0,\qd)$.  In this and the next sections, we
give   examples of such objects using the technique of ambidextrous
objects introduced in \cite{GPT}. We first briefly  recall this
technique.

\subsection{Ambidextrous objects}\label{S:RibbonGraphInv} Let $\cat$
be a (strict) ribbon Ab-category, i.e., a (strict) pivotal tensor
Ab-category with braiding and twist.  We denote the braiding
morphisms in $\cat$ by $c_{V,W}:V\otimes W \rightarrow W \otimes V$
and the duality morphisms $b_V, d_V, b'_V, d'_V$ as in
Section~\ref{S:SphCat}.  We   assume   that the ground ring $K$ of
$\cat$ is a field.

For an  object $J$ of $\cat$ and an  endomorphism $f$ of $J\otimes
J$, set
$$
\tr_{L}(f)=(d_{J}\otimes \Id_{J})\circ(\Id_{J^{*}}\otimes
f)\circ(b'_{J}\otimes \Id_{J}) \in \End(J),
$$
$$
\tr_{R}(f)=(\Id_{J}\otimes d'_{J}) \circ (f \otimes \Id_{J^{*}})
\circ(\Id_{J}\otimes b_{J}) \in \End(J).
$$
An object $J$ of $\cat$ is   \emph{ambidextrous} if
$\tr_{L}(f)=\tr_R(f)$ for all $f \in\End(J \otimes J).$

Let $Rib_\cat$ be the category of $\cat$-colored ribbon graphs and
let $F:Rib_\cat\to \cat $ be the usual ribbon functor   (see
\cite{Tu}).  Let $T_{V}$ be a $\cat$-colored (1,1)-ribbon graph
whose open string is oriented downward and colored with a simple
object $V$ of $\cat$. Then $F(T_{V})\in \End_{\cat}(V)= K \Id_V$.
Let $<T_{V}> \, \in K$ be such that $F(T_{V})= \, <T_{V}> \, \Id_V$.
For any   objects $V, V'$ of $\cat$ such that $V'$ is simple, set
$$
S'(V,V')=\left< \epsh{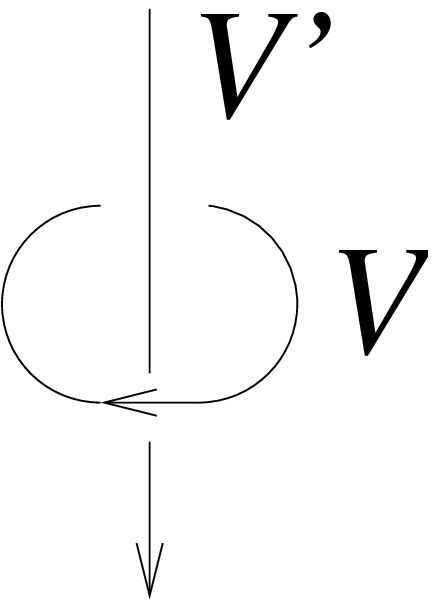}{10ex}\right> \, \in K\, .
$$

Fix basic data $\{V_i, w_i \}_{i \in I}$ in $\cat$ and  a simple
ambidextrous object $J$ of $\cat$. Set
\begin{equation}\label{E:DefdJ--}
 I_0={I_0}(J)=\{ i\in I :   S'(J,V_i)
\neq 0\}\, .
\end{equation}
  For $i\in I_0$, set
\begin{equation}\label{E:DefdJ}
\qd_J(i)= \frac{S'(V_i,J)}{S'(J,V_i)}\in \FK\, .
\end{equation}
We view $\qd_J(i)$ as the  modified quantum dimension of $V_i$
determined by $J$.

\begin{theorem}[\cite{GPT}]\label{eee} Let $L$ be a
$\cat$-colored ribbon $(0,0)$-graph  having  an edge   $e$
  colored with   $V_i$ where   $i\in {  I}_0$.  Cutting $L$ at $e$, we obtain a
  colored ribbon (1,1)-graph $T^e$ whose closure is $L$.  Then the
  product
$  {   \qd}_J(i) <T^{e}> \, \in \FK
 $
  is independent of the choice of $e$ and yields an isotopy
  invariant of $L$.
\end{theorem}

\begin{corollary}\label{C:ribbonambi}
  The pair $(I_0=I_0(J),\qd_J)$ is a t-ambi pair in $\cat$.
\end{corollary}

\subsection{Example: The standard quantum
$6j$-symbols}\label{SS:QSLA} Let $\g$ be a simple Lie algebra over
$\C$.  Let $q$ be a primitive complex root of unity of order $2r$, where $r$
is a positive integer.  Let $\Uqg$ be the Drinfeld-Jimbo $\C$-algebra
associated to $\g$.  This algebra is presented by the generators $E_k, F_k,
K_k, K_k^{-1}$, where $k=1,{\ldots},m$ and the usual relations.  Let $\Uqghat$
be the Hopf algebra obtained as the quotient of $\Uqg$ by the two-sided ideal
generated by $E_k^r, F_k^r, K_k^r-1$ with $k=1,{\ldots},m$.  It is known that
$\Uqghat$ is a finite dimensional ribbon Hopf algebra.  The category of
$\Uqghat$-modules of finite complex dimension is a ribbon Ab-category.  It
gives rise to a ribbon Ab-category $\cat$ obtained by annihilating all
negligible morphisms (see Section XI.4 of \cite{Tu}).

Let $I$ be the set of weights belonging to the Weyl alcove determined by $\g$
and $r$.  For $i\in I$, denote by $V_i$ the simple weight module with highest
weight $i$.  By \cite{TW}, there is an involution $I\to I, i \rightarrow i^*$
and morphisms $\{w_i:V_i \to V_{i^*}^*\}_{i \in I}$ satisfying Equation
\eqref{E:FamilyIso}.  Then $\cat$ is a modular category with basic data
$\{V_i\}_{i\in I}$.

Let $J=\C$ be the unit object of $\cat$.  By Lemma 1 of \cite{GPT}, the object
$J$ is ambidextrous and by Corollary \ref{C:ribbonambi} the pair
$(I_0={I_0}(J),\qd_J)$ is a t-ambi pair.  The general theory of \cite{GPT}
implies that ${ I}_0=I$ and ${ \qd}_J $ is the usual quantum dimension.

We can apply the techniques of Sections \ref{S:Modified6j} and \ref{S:Prop} to
the basic data $\{V_i, \, w_i \}_{i \in I}$, and the t-ambi pair $(
I={I_0}(J),\qd_J)$. It is easy to check that the corresponding modified
$6j$-symbols are the usual quantum $6j$-symbols associated to $\g$. Note that
here all pairs $(i,j)\in I^2$ are good.

\section{Example: Quantum $6j$-symbols from $\UsltH$ at roots of
  unity} \label{SS:UsltH}

In this section we consider a category of modules over the quantization
$\Ubar$ of $\slt$ introduced in \cite{GPT}.  The usual quantum dimensions and
the standard $6j$-symbols associated to this category are generically zero. We
equip this category with basic data and a t-ambi pair leading to non-trivial
$6j$-symbols.

Set $q=\e^{ki\pi/\ro}\in \C$, where $\ro$ is a positive integer and $k$ is an
odd integer coprime with $ \ro$.  We use the notation $q^x$ for $\e^{x
  ki\pi/\ro}$, where $x\in \C$ or $x$ is an endomorphism of a finite
dimensional vector space.

Let $\Uslt$ be the standard quantization of $\slt$, i.e. the $\C$-algebra with
generators $E, F, K, K^{-1}$ and the following defining relations:
\begin{equation}\label{E:RelUslt}
  KK^{-1} =K^{-1}K=1,  \,  KEK^{-1} =q^2E, \,  KFK^{-1}=q^{-2}F,\,
  [E,F] =\frac{K-K^{-1}}{q-q^{-1}}.
\end{equation}
This algebra   is a Hopf algebra with   coproduct $\Delta$, counit
$\varepsilon$, and antipode $S$   defined by the formulas
\begin{align*}
  \Delta(E)&= 1\otimes E + E\otimes K, & \Delta(F)&=K^{-1} \otimes F +
  F\otimes 1, & \Delta(K^{\pm 1})&=K^{\pm 1} \otimes K^{\pm 1},\\
  \varepsilon(E)&= \varepsilon(F)=0, &
  \varepsilon(K)&=\varepsilon(K^{-1})=1,\\
  S(E)&=-EK^{-1}, & S(F)&=-KF, & S(K)&=K^{-1}.
\end{align*}

 Let $\UsltH$ be
the $\C$-algebra given by the generators $E, F, K, K^{-1}, H$,
relations
  \eqref{E:RelUslt}, and the following additional    relations:
\begin{align*}
  HK&=KH, & HK^{-1}&=K^{-1}H, & [H,E]&=2E, & [H,F]&=-2F.
\end{align*}
The algebra $\UsltH$ is a Hopf algebra with coproduct $\Delta$, counit
$\varepsilon$, and antipode $S$ defined as above on $E,F,K,K^{-1}$ and defined
on $H$ by the formulas
\begin{align*}
  \Delta(H)&=H\otimes 1 + 1 \otimes H, & \varepsilon(H)&=0, &
  S(H)&=-H.
\end{align*}

Following \cite{GPT}, we define $\Ubar$ to be the quotient of $\UsltH$ by the
relations $E^{\ro}=F^{\ro}=0$.  It is easy to check that the operations above
turn $\Ubar$ into a Hopf algebra.

Let $V$ be a $\Ubar$-module.  An eigenvalue $\lambda\in \C$ of the operator
$H:V\to V$ is called a \emph{weight} of $V$ and the associated eigenspace
$E_\lambda(V)$ is called a \emph{weight space}.  We call $V$ a \emph{weight
  module} if $V$ is finite-dimensional, splits as a direct sum of weight
spaces, and $q^H=K$ as operators on $V$.

Given two weight modules $V$ and $W$, the operator $H$ acts as $H\otimes 1 + 1
\otimes H$ on $V\otimes W$. So, $E_\lambda(V) \otimes E_{\mu}(W) \subset
E_{\lambda +\mu}(V\otimes W)$ for all $\lambda, \mu\in \C$.  Moreover,
$q^{\Delta(H)}=\Delta(K)$ as operators on $V\otimes W$.  Thus, $V\otimes W$ is
a weight module.

Let $\catd$ be the tensor Ab-category of weight $\Ubar$-modules. By Section
6.2 of \cite{GPT}, $\catd$ is a ribbon Ab-category with ground ring $\C$.  In
particular, for any object $V$ in $\catd$, the dual object and the duality
morphisms are defined as follows: $V^* =\Hom_\C(V,\C)$ and
\begin{align}\label{E:DualityForCat}
  b_{V} :\, & \C \rightarrow V\otimes V^{*} \text{ is given by } 1 \mapsto
  \sum
  v_j\otimes v_j^*,\notag\\
  d_{V}:\, & V^*\otimes V\rightarrow \C \text{ is given by }
  f\otimes w \mapsto f(w),\notag\\
  d_{V}':\, & V\otimes V^{*}\rightarrow \C \text{ is given by } v\otimes f
  \mapsto f(K^{1-{\ro}}v),\notag
  \\
  b_V':\, & \C \rightarrow V^*\otimes V \text{ is given by } 1 \mapsto \sum
  v_j^*\otimes K^{{\ro}-1}v_j,
\end{align}
where $\{v_j\}$ is a basis of $V$ and $\{v_j^*\}$ is the dual basis of $V^*$.

For an isomorphism classification of simple weight $\Uslt$-modules (i.e.,
modules on which $K$ acts diagonally), see for example \cite{Kas}, Chapter VI.
This classification implies that simple weight $\Ubar$-modules are classified
up to isomorphism by highest weights. For $i\in \C$, we denote by $V_{i}$ the
simple weight $\Ubar$-module of highest weight $i+{\ro}-1$. This notation
differs from the standard labeling of highest weight modules. Note that
$V_{-{\ro}+1}=\C$ is the trivial module and $V_0$ is the so called Kashaev
module.

We now define basic data in $\catd$.  Set $I=\C$ and define an involution
$i\mapsto i^*$ on $I$ by $i^*= i$ if $i \in \Z$ and $i^*=-i$ if $i\in
\C\setminus \Z$.

\begin{lemma}\label{L:w_iUsltH} There are isomorphisms
  $\{w_{i}:V_{i} \rightarrow (V_{{i}^*})^*\}_{{i} \in I}$
  satisfying~\eqref{E:FamilyIso}.
\end{lemma}
\begin{proof}
  If $i\in \C\setminus \Z$, then $ V_{i}^* \cong V_{-i} $, see \cite{GPT} . We
  take an arbitrary isomorphism $V_{i} \rightarrow (V_{{i}^*})^*=V_{-i}^*$ for
  $w_{i}$ and choose $w_{{i}^*}$ so that it satisfies \eqref{E:FamilyIso}.  If
  $i\in\Z$, then $i=i^*$ and there is a unique integer $d $ such that $1\leq
  d\leq {\ro}$ and $d\equiv i$ (mod ${\ro}$).  Let $v_0$ be a highest weight
  vector of $V_i$. Set $v_j=F^jv_0$ for $j=1,{\ldots},d-1$.  Then
  $\{v_0,{\ldots},v_{d-1}\}$ is a basis of $V_i$ and in particular, $\dim V_i=
  d$.  Let $\{v_j^*\}$ be the basis of $V_i^*$ dual to $\{v_j\}$ so that
  $v_j^*(v_k)=\delta_{j,k}$.  Define an isomorphism $w_i: V_{i} \rightarrow
  (V_{{i}})^*$ by $v_0 \mapsto v_{d-1}^*$.  To verify \eqref{E:FamilyIso}, it
  suffices to check that the morphism $(w_i^{-1})^* w_i:V_i \rightarrow
  V_i^{**}=V_i$ is multiplication by $K^{1-{\ro}}$.  To do this it is enough
  to calculate the image of $v_0$.  A direct calculation shows that
  $$
  S(F)^{d-1}v_0=(-1)^{d-1}q^{-(d-1)}v_{d-1}=q^{{\ro}(d-1)-(d-1)}v_{d-1}.
  $$
  It follows that $F^{d-1}v^*_{d-1}=q^{({\ro}-1)(d-1)}v_0^*$. 
  Thus for $k=0,{\ldots},d-1$,
  $$
  (w^{-1}_i)^*(v_{d-1}^*)(v_k^*)=v_{d-1}^*(w^{-1}_i(v_k^*))= \delta_{k,0}\,
  q^{-({\ro}-1)(d-1)}.
  $$
  In other words, $(w^{-1}_i)^*(v_{d-1}^*)= q^{-({\ro}-1)(d-1)}v_0^{**}$.
  Under the standard identification $V_i^{**}=V_i$, we obtain
  $$
  (w_i^{-1})^* w_i(v_0)=(w_i^{-1})^*(v_{d-1}^*)=q^{-({\ro}-1)(d-1)}v_0.
  $$
  The proof is completed by noticing that
  $K^{1-{\ro}}v_0=q^{-({\ro}-1)(d-1)}v_0.$
\end{proof}

Set $I_0=(\C \setminus \Z) \cup \ro\Z\subset I=\C$. We say that a simple
weight $\Ubar$-module $V_i$ is \emph{typical} if $i\in I_0$.  By \cite{GPT},
every typical module $J=V_i$ is ambidextrous and the set $I_0(J)$ defined by
\eqref{E:DefdJ--} is equal to $I_0$.  Formula \eqref{E:DefdJ} defines a
function $\qd_J:I_0\rightarrow \C$. As shown in \cite{GPT}, up to
multiplication by a non-zero complex number, $\qd_J$ is equal to the function
$\qd:I_0\rightarrow \C$ defined by
\begin{equation}\label{E:dUsltH}
  \qd(k)
  =\frac{1}{\prod_{j=0}^{{\ro}-2}\qn{k-j-1}} \quad {\text {for}}\quad k\in I_0
\end{equation}
(here $\qn a=q^a-q^{-a}$ for all $a\in \C$). Note in particular that
\begin{equation}\label{E:dUsltH+}\qd(k)=\qd({k+2{\ro}})\end{equation} for all
$k\in I_0$. Theorem \ref{eee} implies the following lemma.
\begin{lemma}\label{L:tambipair}
  $(I_0,\qd)$ is a t-ambi pair in $\catd$.
\end{lemma}

We can apply the techniques of Sections~\ref{S:Modified6j} and \ref{S:Prop} to
the category $\catd$ with basic data $\{V_i, w_{i} \}_{{i} \in I}$ and the
t-ambi pair $(I_0,\qd)$.  This yields a $6j$-symbol $\sjtop ijklmn$ for all
indices $i,j,k,l,m,n\in I=\C$ with at least one of them in $I_0$.  In the case
$i,j,k,l,m,n\in \C\setminus\Z$ this $6j$-symbol is identified with a complex
number: for any numbers $i,j,k \in \C\setminus\Z$ we have that $\dim
H(i,j,k)=0$ or $\dim H(i,j,k)=1$. In the last case, there is a natural choice
of an isomorphism $H(i,j,k)\simeq \C$. Hence the $6j$-symbols can be viewed as
taking value in $\C$ and can be effectively computed via certain recurrence
relations forming a ``skein calculus'' for $\catd$-colored ribbon trivalent
graphs. These relations and the resulting explicit formulas for the
$6j$-symbols will be discussed in \cite{Todo}. We formulate here one
computation.

\begin{example}\label{Ex:compute}
  For $\ro=3$ and $q=\e^{i\pi/3}$ we can compute $\sjtop ijklmn$ for
  all
  $i,j,k,l,m,n\in \C\setminus\Z$ as follows.
  \begin{enumerate}
\item If
  $k=i+j+2$, $l=n-j-2$, and $m=n+i+2$, then
  $$\sjtop ijklmn=\qn{n+1}\qn{n+2}.$$
\item If
  $k=i+j$, $l=n-j$, and $m=n+i+2$, then
  $$\sjtop ijklmn=\qn{j+2}\qn{n+1}.$$
\item If
  $k=i+j$, $l=n-j$, and $m=n+i$, then
  $$\sjtop ijklmn=-\left(q^{i+j-n}+q^{i-j+n}+q^{-i+j+n}+q^{i-j-n}
    +q^{-i+j-n}+q^{-i-j+n}\right).$$
\item For tuples $i,j,k,l,m,n\in \C\setminus\Z$ obtained from the tuples in
  (1), (2), (3) by tetrahedral permutations, the value of the $6j$-symbol is
  computed by (1), (2), (3). In all other cases
  $$\sjtop ijklmn=0.$$
\end{enumerate}
\end{example}

Note finally that in the present setting the admissibility condition of
Section \ref{S:Prop} is generically satisfied.


\section{Example: Quantum $6j$-symbols from $\Uslt$ at roots of
  unity} \label{SS:Uslt6j} In this section we
construct a category of $\Uslt$-modules which is pivotal but {\it a
priori} not ribbon. We use the results of Section \ref{SS:UsltH} to
provide this category with basic data and a t-ambi pair.

Let $\ro$, $q$, and $\Uslt$ be as in Section \ref{SS:UsltH}. Assume
additionally that $\ro$ is odd and set $\m=\frac{\ro-1}{2}$. Let
$\Utilde$ be the Hopf algebra obtained as the quotient of $\Uslt$
  by the relations $E^{\ro}=F^{\ro}=0$.

A  $\Utilde$-module  $V$ is a \emph{weight module} if $V$ is finite
dimensional and $K$ act diagonally on $V$.  We say that $v\in V$ is
a \emph{weight vector} with \emph{weight} $\tilde\lambda\in
\C/2{\ro}\Z$ if $Kv=q^{\tilde\lambda}v$. Let $\cat$ be the tensor
Ab-category of   weight $\Utilde$-modules (the ground ring is $\C$).
\begin{lemma}
  The category $\cat$ is a pivotal Ab-category with duality morphisms
  \eqref{E:DualityForCat}.
\end{lemma}
\begin{proof}
  It is a straightforward calculation to show that the left and right duality
  are compatible and   define a pivotal structure on $\cat$.
\end{proof}

 If $V$ is a   weight
$\Utilde$-module containing a weight vector $v$ of weight
$\tilde\lambda$ such that $Ev=0$ and $V= \Uslt v$, then   $V$ is a
\emph{highest weight module} with \emph{highest weight}
$\tilde\lambda$.  The classification of simple weight
$\Uslt$-modules (see Chapter VI of \cite{Kas}) implies that highest
weight $\Utilde$-modules are classified up to isomorphism by their
highest weights.  For any $\tilde i\in \C/2{\ro}\Z$, we denote  by
$V_{\tilde i}$ the highest weight module with highest weight
$\widetilde{i}+ \widetilde {\ro-1}$, where $\widetilde {\ro-1}\in
\C/2{\ro}\Z$ is the projection of $r-1$ to $\C/2{\ro}\Z$. The
following lemma yields basic data in $\cat$.

\begin{lemma}\label{L:w_tildei} Set $I^\cat=(\C\setminus \Z)/2{\ro}\Z $ and
  define an involution $\tilde i\mapsto {\tilde i}^*$ on $I^\cat$ by
  $\tilde{i}^*=-\tilde i$.  Then there are isomorphisms $\{w_{\tilde
    i}:V_{\tilde i} \rightarrow (V_{\tilde{i}^*})^*\}_{\tilde{i} \in I^\cat}$
  satisfying Equation~\eqref{E:FamilyIso}.
\end{lemma}
\begin{proof}
  The lemma follows as in the proof of Lemma \ref{L:w_iUsltH}.
\end{proof}

\begin{theorem}\label{rol} Let $\qd: I^\cat\to  \C$ be the function
\begin{equation}\label{E:dUslt}
  \qd( {\tilde
    k})=\frac{1}{\prod_{j=0}^{{\ro}-2}
    \qn{\widetilde{k}-\widetilde{1}-\widetilde{j}}}
\end{equation}
for any $\tilde k \in I^\cat$. Then the pair $(I_0=I^\cat, \qd)$ is
t-ambi in $\cat$.
\end{theorem}

A proof of Theorem \ref{rol} will be given at the end of the section using the
following preliminaries. Recall the category $\catd=\Ubar$-mod from Section
\ref{SS:UsltH}. Let $\varphi: \catd \rightarrow \cat$ be the functor which
forgets the action of $H$.  Consider the functors $G_\cat$ and $G_\catd$
associated to $\cat$ and $\catd$, respectively (see Section
\ref{S:InvOfGraphs}). The functor $\varphi: \catd \rightarrow \cat$ induces
(in the obvious way) a functor $\varphi_{\Graph}:\Graph_\catd \rightarrow
\Graph_\cat$ such that the following square diagram commutes:
\begin{equation}\label{E:CommDiag}
  \xymatrix 
  {
    \Graph_\catd  \ar[r]^{\varphi_{\Graph}}  \ar[d]_{G_\catd} &  \Graph_\cat
    \ar[d]^{G_\cat}
    \\
    \catd \ar[r]_{\varphi} & \cat
  }
\end{equation}

We say that the highest weight $\Utilde$-module $V_{\tilde i}$ is
\emph{typical} if $\tilde i \in I^\cat \cup \{\tilde{0},\tilde{r}\}$. In other
words, $V_{\tilde i}$ is typical if its highest weight is in $I^\cat \cup
\{\widetilde{-1},\widetilde{{2\m}}\}\subset \C/2{\ro}\Z$.  A typical module
$V_{\tilde i}$ has dimension ${\ro}$ and its weight vectors have the weights
$\widetilde{i}+ 2\widetilde{k}$ where $k=-\m,-\m+1,{\ldots},\m$.  It can be
shown that if a typical module $V$ is a sub-module of a finite dimensional
$\Utilde$-module $W$, then $V$ is a direct summand of $W$.  Therefore, if
$V_{\tilde i} $ and $V_{\tilde j}$ are typical modules such that $\tilde
i+\tilde j\notin \Z/2{\ro}\Z$, then
\begin{equation}\label{E:TensorProdDe}
V_{\tilde i}\otimes V_{\tilde j}\cong \bigoplus_{k=-{\m}}^{\m}
V_{\widetilde{i}+ \widetilde{j}+ 2\widetilde{k}}.
\end{equation}

As in Section \ref{SS:UsltH}, for any $i\in \C$, we denote by $V_i$ a weight
module in $\catd$ of highest weight $i+ {\ro}-1$. Clearly, $\varphi(V_i)\simeq
V_{\tilde i}$ for all $i\in \C$, where $\tilde i= i \, ({\rm {mod}}\, 2r\Z)$
and $\simeq$ denotes isomorphism in $\cat$. To specify an isomorphism
$\varphi(V_i)\simeq V_{\tilde i}$, we fix for each $i\in \C$ a highest weight
vector in $V_i\in \catd$ and we fix for each $a\in \C/2r\Z$ a highest weight
vector in $V_a\in \cat$. Then for every $i\in \C$, there is a canonical
isomorphism $ \varphi(V_i)\to V_{\tilde i}$ in $\cat$ carrying the highest
weight vector of $V_i$ into the highest weight vector of $V_{\tilde i}$. To
simplify notation, we write in the sequel $ \varphi(V_i)= V_{\tilde i}$ for
all $i\in \C$.

We say that a triple $(i,j,k)\in \C^3$ has \emph{{\charge}} $i+j+k\in\C$. For
$i,j,k\in \C$, set $$H^{ijk} = \Hom_{\catd} (\unit, V_i\otimes V_j\otimes V_k)
\quad {\text{and}}\quad H^{ij}_k = \Hom_{\catd} (V_k, V_i\otimes V_j ).$$

\begin{lemma}\label{L:CanIso}
  Let $i,j,k\in \C\setminus \Z$. If the {\charge} of $(i,j,k)$ does not belong
  to the set $\{-2{\m},-2{\m}+2,{\ldots},2{\m}\}$, then $H^{ijk}=0$. If the
  {\charge} of $(i,j,k)$ belongs to the set
  $\{-2{\m},-2{\m}+2,{\ldots},2{\m}\}$, then the composition of the
  homomorphisms
  $$
  H^{ijk} =\Hom_{\catd} (\unit, V_i\otimes V_j\otimes V_k) \stackrel{\varphi}
  \longrightarrow \Hom_{\cat} (\unit, \varphi(V_i)\otimes \varphi(V_j) \otimes
  \varphi(V_k))
  $$
  $$
  = \Hom_{\cat} (\unit, V_{\tilde i}\otimes V_{\tilde j} \otimes V_{\tilde k})
  $$ is an isomorphism.
\end{lemma}
\begin{proof}
  We begin with a simple observation.  Since $i\in \C \setminus \Z$, the
  character formula for $V_i$ is $\sum_{l=-\m}^{\m}u^{i+2l}$ where the
  coefficient of $u^a$ is the dimension of the $a$-weight space.  Therefore,
  the character formula for $V_i\otimes V_j$ is
  \begin{equation}
    \label{E:tensorViVj}
    \left(\sum_{l=-\m}^{\m}u^{i+2l}\right)\left(\sum_{m=-\m}^{\m}u^{j+2m}\right)
    =\sum_{l,m=-\m}^{\m}u^{i+j+2l+2m}.
  \end{equation}

  As we know, Formula \eqref{E:Hiso} defines an isomorphism $H^{ijk} \simeq
  H^{ij}_{-k}$. We claim that $\dim( H^{ij}_{-k})=1$ if the {\charge} of
  $(i,j,k)$ belongs to $\{-2{\m},-2{\m}+2,{\ldots},2{\m}\}$ and $H^{ij}_{-k}
  =0$, otherwise.  To see this, we consider two cases.

  Case 1: $i+j\in \Z$.  Since $k\notin \Z$, the {\charge} of $(i,j,k)$ is not
  an integer and does not belong to the set
  $\{-2{\m},-2{\m}+2,{\ldots},2{\m}\}$.  Equation~\eqref{E:tensorViVj} implies
  that all the weights of $V_i\otimes V_j$ are integers. Since $k\notin \Z$,
  we have $H^{ij}_{-k}=0$.

  Case 2: $i+j\notin \Z$.  It can be shown that if $V\in {\catd}$ is a typical
  module which is a sub-module of a module $W\in {\catd}$, then $V$ is a
  direct summand of $W$.  Combining with Equation \eqref{E:tensorViVj} we
  obtain that
  \begin{equation}
    \label{E:tensorDecomViVj}
    V_{i}\otimes V_{j}\simeq \bigoplus_{l=-{\m}}^{\m}
    V_{{i+j+2l}}.
  \end{equation}
  Therefore $H^{ij}_{-k}\neq 0$ if and only if $i+j+k\in
  \{-2{\m},-2{\m}+2,{\ldots},2{\m}\}$.  In addition, Formula
  \eqref{E:tensorDecomViVj} implies that if $H^{ij}_{-k}\neq 0$ then
  $\dim(H^{ij}_{-k})=1$.  This proves  the claim above and the first
  statement of the lemma.

  To prove the second part of the lemma, assume that the {\charge} of
  $(i,j,k)$ is in $\{-2{\m},-2{\m}+2,{\ldots},2{\m}\}$.  It is enough to show
  that the homomorphism
  $$\Hom_{\catd} (\unit, V_i\otimes V_j\otimes V_k) \stackrel{\varphi}  \to
  \Hom_{\cat} (\unit, \varphi(V_i)\otimes \varphi(V_j)\otimes \varphi(V_k))$$
  is an isomorphism.  This homomorphism is injective by the very definition of
  the forgetful functor.  So it suffices to show that the domain and the range
  have the same dimension.  From the claim above, $\dim(H^{ij}_{-k})=1$ and
  thus the domain is one-dimensional.  The assumption on the {\charge} of
  $(i,j,k)$ combined with Equation \eqref{E:TensorProdDe} implies that $\dim
  \Hom_\cat(\varphi(V_{-k}),\varphi(V_i)\otimes \varphi(V_j))=1$. Thus, the
  range is also one-dimensional.  This completes the proof of the lemma.
\end{proof}

We say that a triple $ \tilde i,\tilde j, \tilde k \in \C/2{\ro}\Z $ has
\emph{integral \charge} if $\tilde i+\tilde j+\tilde k\in \Z/2{\ro}\Z$ and
$\tilde i+\tilde j+\tilde k$ is even.  In this case the \emph{\charge} of the
triple $(\tilde i,\tilde j, \tilde k)$ is the unique (even) $n\in
\{-2{\m},-2{\m}+2,{\ldots},2{\m}\}$ such that $\tilde i+\tilde j+\tilde k= n\,
({\rm {mod}}\, 2r)$.

By an ${{I^\cat}}$-colored ribbon graph we mean a $\cat$-colored ribbon graphs
such that all edges are colored with modules $ V_{\tilde i} $ with $\tilde
i\in I^\cat=(\C\setminus \Z)/2{\ro}\Z $. To each ${{I^\cat}}$-colored
trivalent coupon in $S^2$ we assign a triple of elements of $I^\cat $: each
arrow attached to the coupon contributes a term $\varepsilon \tilde i\in
I^\cat$ where $V_{\tilde i}$ is the color of the arrow and $\varepsilon=+1$ if
the arrow is oriented towards the coupon, and $\varepsilon=-1$. Such a coupon
has \emph{integral \charge} if the associated triple has integral \charge. We
say that a ${{I^\cat}}$-colored trivalent ribbon graph in $S^2$ has
\emph{integral \charge s} if all its coupons have integral \charge.  The
\emph{total {\charge}} of an ${{I^\cat}}$-colored trivalent ribbon graph with
integral {\charge s} is defined to be the sum of the {\charge}s of all its
coupons. This \emph{total {\charge}} is an even integer.

\begin{lemma}\label{L:integral}
  Let $T$ be a ${{I^\cat}}$-colored trivalent ribbon graph in $S^2$.  If $T$
  has a coupon which does not have integral {\charge}, then
  $G_\cat\left(T_{V}\right)=0$ for any cutting presentation $T_{V}$ of $T$.
\end{lemma}
\begin{proof}
  Let $(\tilde i, \tilde j, \tilde k)$ be the triple associated with a coupon
  of $T$ which does not have integral {\charge}.  An argument similar to the
  one in the proof of Lemma \ref{L:CanIso} shows that $\Hom_{\cat} (\unit,
  V_{\tilde i}\otimes V_{\tilde j}\otimes V_{\tilde k})=0$.  This forces
  $G_\cat\left(T_{V}\right)=0$ for any cutting presentation $T_{V}$ of $T$.
\end{proof}

\begin{lemma}\label{L:LiftsTotChargeZero}
  Let $T$ be a connected ${{I^\cat}}$-colored trivalent ribbon graph in $S^2$
  with no inputs and outputs and with integral {\charge s}.  There exists a
  $\catd$-colored trivalent ribbon graph $\check{T}$ such that
  $\varphi_{\Graph}(\check{T})= T$ if and only if the total {\charge} of $T$
  is zero.  If the total {\charge} of $T$ is non-zero, then
  $G_\cat\left(T_{V}\right)=0$ for any cutting presentation $T_{V}$ of $T$.
\end{lemma}
\begin{proof} Given a CW-complex space $X$ and an abelian group $G$, we denote
  by $C_n(X;G)$ the abelian group of cellular $n$-chains of $X$ with
  coefficients in $G$.

  Consider a connected ${{I^\cat}}$-colored trivalent ribbon graph $S$ with
  integral {\charge s} and possibly with inputs and outputs.  Let $|S|$ be the
  underlying 1-dimensional CW-complex of $S$ and $|\partial S| \subset |S|$ be
  the set of univalent vertices (corresponding to the inputs and outputs of
  $S$).  The coloring of $S$ determines a 1-chain $\tilde c\in
  C_1(|S|;\C/2{\ro}\Z)$.  Consider also the $0$-chain $$c_{\partial}=\sum_{x
    \in |\partial S|} C_x x\in C_0(|\partial S|;\C/2{\ro}\Z),$$ where $C_x\in
  \C/2\ro\Z$ is the label of the only edge of $S$ adjacent to $x$ if this edge
  is oriented towards $x$ and minus this label otherwise.  The {\charge}s of
  the coupons of~$S$ form a $0$-chain $w\in C_0(|S|\setminus |\partial
  S|;2\Z)$.  Clearly,
  \begin{equation}
    \label{E:dcoloring}
    \partial \tilde c= c_{\partial}+ w\, ({\text {mod}}\, 2\ro \Z).
  \end{equation}
  Suppose that the $0$-chain $c_{\partial}$ lifts to a certain $0$-chain $b\in
  C_0(|\partial S|;\C)$ such that $[b+w]=0\in H_0(|S|;\C)=\C$.  We claim that
  then the 1-chain $\tilde c$ lifts to a 1-chain $c\in C_1(|S|,\C)$ such that
  $\partial c= b+w$.  Indeed, pick any $c'\in C_1(|S|;\C)$ such that $\tilde
  c=c'\, ({\text {mod}}\, 2\ro \Z) $.  Set $x=\partial c' -b -w \in C_0(|S|;
  \C)$. By \eqref{E:dcoloring}, $x\in C_0(|S|;2{\ro}\Z)$.  Clearly,
  $[x]=-[b+w]=0\in H_0(|S|;2{\ro}\Z)$. Then $x=\partial\delta$ for some
  $\delta\in C_1(|S|;2{\ro}\Z)$ and $c=c'-\delta$ satisfies the required
  conditions.

  Let us now prove the first statement of the lemma. By assumption, $|\partial
  T| =\emptyset$. As above, the coloring of $T$ determines a 1-chain $\tilde
  c\in C_1(|T|;\C/2{\ro}\Z)$. Let $w\in C_0(|T| ;2\Z)$ be the $0$-chain formed
  by the {\charge}s of the coupons of $T$.  The total {\charge} of $T$ is
  $[w]\in H_0(|T|;2\Z)= 2\Z$. If $[w]=0 $, then the preceding argument (with
  $b=0$) shows that $\tilde c$ lifts to a 1-chain $c\in C_1(|T|;\C)$ such that
  $\partial c=w$.  The chain $c$ determines a coloring of all edges of $T$ by
  the modules $\{V_i\, \vert \, i\in \C\}$.  Clearly, the {\charge} of any
  coupon is equal to the corresponding value of $ w$. In particular, all these
  {\charge}s belong to the set $\{-2\ro', -2\ro'+2, \ldots, 2\ro'\}$. By Lemma
  \ref{L:CanIso}, this coloring of edges can be uniquely extended to a
  $\catd$-coloring $\check{T}$ of our graph such that
  $\varphi_{\Graph}(\check{T})= T$.  Conversely, if
  $\varphi_{\Graph}(\check{T})= T$ then with the same notation $w=\partial c$
  and thus the total {\charge} of $T$ is zero.

  Suppose now that $[w]\neq 0$.  Consider the cutting presentation $T_{V}$
  where $V=V_{\tilde\alpha}$ with $\tilde \alpha\in \C \setminus \Z \, ({\text
    {mod}}\, 2\ro \Z) $ is the color of an edge of $T$.  Here $|\partial
  T_{V}|$ is the two-point set formed by the extremities of $T_{V}$ and
  $c_{\partial}$ is the $0$-chain $\tilde\alpha \times \! {\text{ (the input
      vertex minus the output vertex)}}$. As $$[ w]=[\partial \tilde
  c]-[c_{\partial}]=0\in H_0(|T_{V}|;\C/2\ro \Z),$$ we have $[w]\in 2\ro \Z$.
  Pick $\alpha\in\C \setminus\Z$ such that $ \tilde\alpha =\alpha \, ({\text
    {mod}}\, 2\ro \Z) $.  Consider the $0$-chain
  $$b= \alpha \times \! {\text{ (the input vertex minus  the output
      vertex)}} - [w] \times \! {\text{ (the output vertex)}}.$$ Clearly,
  $[w+b]=0\in H_0(|T_{V}|;\C)$.  The argument at the beginning of the proof
  and Lemma~\ref{L:CanIso} imply that there is
  $\check{T}\in\Hom_{\Graph_\catd}(V_\alpha,V_{\alpha+[w]})$ such that $
  \varphi_{\Graph}(\check{T})= T_{V}$.  But then $G_\catd(\check{T})=0$
  because it is a morphism between non-isomorphic simple modules. Therefore
  $G_\cat(T_{V})= G_\cat(\varphi_{\Graph}(\check{T}))
  =\varphi(G_{\catd}(\check{T}))=0$.
\end{proof}

\subsection{Proof of Theorem \ref{rol}}
Let $T_{V_{\tilde i}}$ and $T_{V_{\tilde j}}$ be cutting presentations of a
connected ${{I^\cat}}$-colored trivalent ribbon graph $T$ in $S^2$ with no
inputs and outputs.  We claim that \begin{equation}\label{E:dUslt+} \qd(
  {\tilde i})<G_\cat(T_{V_{\tilde i}})> \, =\qd( {\tilde
    j})<G_\cat(T_{V_{\tilde j}})>.
\end{equation}
By the previous two lemmas it is enough to consider the case where all coupons
of $T$ have integral {\charge} and the total {\charge} of $T$ is zero.  By
Lemma \ref{L:LiftsTotChargeZero}, there is a $\catd$-colored trivalent ribbon
graph $\check{T}$ such that $\varphi_{\Graph}(\check{T})= T$. Cutting
$\check{T}$ at the same edges, we obtain $\catd$-colored graphs
$\check{T}_{V_{i}} $ and $\check{T}_{V_{j}} $ carried by $\varphi_{\Graph}$ to
$ T_{V_{\tilde i}}$ and $ T_{V_{\tilde j}}$, respectively. From
Diagram~\eqref{E:CommDiag} we have $<G_\catd(\check{T}_{V_i})> \, =\,
<G_\cat(T_{V_{\tilde i}})>$ and $<G_\catd({\check{T}}_{V_j})>\, = \,
<G_\cat(T_{V_{\tilde j}})>$.  Lemma~\ref{L:tambipair} implies that $\qd(
i)<G_\catd(\check{T}_{V_i})> \, =\qd( j)<G_\catd({\check{T}}_{V_j})>$.  Since
$\qd( { i})=\qd( {\tilde i})$ and $\qd( {j})=\qd( {\tilde j})$, these formulas
imply Formula~\eqref{E:dUslt+}.

\begin{remark} The category
$\cat$ with  basic data $\{V_{\tilde i}, w_{\tilde i}\}_{\tilde i\in
I}$
and t-ambi pair $ ( I^\cat=(\C\setminus\Z)/2\ro\Z,\qd )$ determines modified
$6j$-symbols $\sjtop {\tilde i}{\tilde j}{\tilde k}{\tilde l}{\tilde m}{\tilde
  n}$ for ${\tilde i},{\tilde j},{\tilde k},{\tilde l},{\tilde m},{\tilde n}
\in I^\cat$.  Using Lemma \ref{L:CanIso}, one observes that as in Section
\ref{SS:UsltH}, the $6j$-symbols of this section can be viewed as taking
values in $\C$. The values of these $6j$-symbols are essentially the same as
the values of the $6j$-symbols derived from the category $\catd$.  More
precisely, let $\tilde i,\tilde j,\tilde k,\tilde l,\tilde m,\tilde n\in
I^\cat$. If the total {\charge} of $\Gamma(\tilde i,\tilde j,\tilde k,\tilde
l,\tilde m,\tilde n)$ is non-zero, then $\sjtop {\tilde i}{\tilde j}{\tilde
  k}{\tilde l}{\tilde m}{\tilde n}=0$.  If the total {\charge} of
$\Gamma(\tilde i,\tilde j,\tilde k,\tilde l,\tilde m,\tilde n)$ is zero, then
for some lifts $i,j,k,l,m,n\in\C\setminus\Z$ of $\tilde i,\tilde j,\tilde
k,\tilde l,\tilde m,\tilde n$,
  $$\sjtop {\tilde
    i}{\tilde j}{\tilde k}{\tilde l}{\tilde m}{\tilde n} =\sjtop
  ijklmn_\catd.$$
Example \ref{Ex:compute} computes these $6j$-symboles for $\ro=3$
and
  $q=\e^{i\pi/3}$.
\end{remark}


\section{Three-manifold invariants}\label{S:3Man}

In this section we derive a topological invariant of links in closed
orientable $3$-manifolds from a suitable pivotal tensor Ab-category.  We also
show that Sections \ref{SS:QSLA} and \ref{SS:Uslt6j} yield examples of such
categories.

\subsection{Topological preliminaries} Let $M$ be a closed orientable
$3$-manifold and $L$ a link in $M$. Following \cite{BB}, we use the term
\emph{quasi-regular triangulation} of $M$ for a decomposition of $M$ as a
union of embedded tetrahedra such that the intersection of any two tetrahedra
is a union (possibly, empty) of several of their vertices, edges, and
(2-dimensional) faces.  Quasi-regular triangulations differ from usual
triangulations in that they may have tetrahedra meeting along several
vertices, edges, and faces. Nevertheless, the edges of a quasi-regular
triangulation have distinct ends. A \emph{Hamiltonian link} in a quasi-regular
triangulation $\T$ is a set $\LL$ of unoriented edges of $\T$ such that every
vertex of $\T$ belongs to exactly two edges of $\LL$.  Then the union of the
edges of $\T$ belonging to $\LL$ is a link $L$ in $M$. We call the pair
$(\T,\LL)$ an \emph{$H$-triangulation} of $(M,L)$.

\begin{proposition}[\cite{BB}, Proposition 4.20]\label{L:Toplemma-}
  Any pair (a closed connected orientable $3$-manifold $M$, a non-empty link
  $L\subset M$) admits an $H$-triangulation.
\end{proposition}

\subsection{Algebraic preliminaries} \label{SS:ModifiedTV} Let $\cat$ be a
pivotal tensor Ab-catego\-ry with ground ring $K$, basic data $\{V_i, w_i:V_i
\to V_{i^*}^*\}_{i \in I}$, and t-ambi pair $(I_0,\qd)$. 
As in Section \ref{S:RibbonGraphInv}, we assume that the ground ring $K$ of
$\cat$ is a field.
To define the associated 3-manifold invariant we need the following
requirements on~$\cat$.  Fix an abelian group $\Gr$.  Suppose that $\cat$ is
\emph{$\Gr$-graded} in the sense that for all $g\in\Gr$, we have a class
$\cat_g$ of object of $\cat$ such that
\begin{enumerate}
\item $\unit\in\cat_0$,
\item  if $ V\in\cat_g$, then $V^* \in\cat_{-g}$,
\item if  $ V\in\cat_g,\, V'\in\cat_{g'}$, then  $V\otimes V'\in\cat_{g+g'}$,
\item if $ V\in\cat_g,\, V'\in\cat_{g'}$, and $g\neq g'$, then
  $\Hom_\cat(V,V')=\{0\}$.
\end{enumerate}

For any $g\in G$, set $$I^g=\{i\in I \, \vert\, V_i\in \cat_g\}.$$
We shall assume that  $\Gr$ contains a set $\Gs$ with the following
properties:
\begin{enumerate}
\item $\Gs$ is symmetric: $-\Gs=\Gs$,
\item $\Gr$ can not be covered by a finite number of translated copies of
  $\Gs$, in other words, for any $ g_1,\ldots ,g_n\in \Gr$, we have
  $\bigcup_{i=1}^n (g_i+\Gs) \neq\Gr $,
\item if $g\in\Gr\setminus\Gs$, then the set $I^g$ is a finite subset of
  $I_0$, $\qd(I^g)\subset K^*$, and every object of $\cat_g$ is isomorphic to
  a direct sum of a finite family of objects $\{V_i\, \vert \, i\in I^g\}$.
\end{enumerate}

Note that the last condition and the definition of a basic data
imply that for $g\in G \setminus X$, every simple object of $\cat_g$
 is isomorphic to $V_i$ for a unique $i\in I^g$.

Finally, we  assume to have  a map $\bb:I_0\to K$  such that
\begin{enumerate}
\item  $\bb(i)=\bb(i^*)$, for all $i\in I_0$,
\item  for any $g,g_1,g_2\in\Gr\setminus\Gs$ with $g+g_1+g_2=0$ and for all
  $j\in I^g$,
  $$\bb(j)=\sum_{j_1\in I^{g_1},\, j_2\in I^{g_2}}
  \bb({j_1})\bb({j_2})\dim(H(j,j_1,j_2)).$$
\end{enumerate}
Denoting by $\bb_g$ the formal sum $\bb_g=\sum_{j\in I^g} \bb(j)
V_j$, one obtains $\bb_{g_1}\otimes\bb_{g_2}=\bb_{g_1+g_2}$ whenever
$g_1,g_2,g_1+g_2\in\Gr\setminus\Gs$.  The map $g\mapsto \bb_g$ can
be seen as a ``representation'' of $\Gr \setminus \Gs$ in the
Grothendieck ring of $\cat$.

\subsection{A state sum invariant}\label{TTVV} 
We start from the algebraic data described in the previous subsection and
produce a topological invariant of a triple $(M,L,h)$, where $M$ is a closed
connected oriented 3-manifold, $L\subset M$ is a non-empty link, and $h\in
H^1(M,G)$.

Let $(\T,\LL)$ be an $H$-triangulation of $(M,L)$.  By a \emph{$\Gr$-coloring}
of $\T$, we mean a $\Gr$-valued $1$-cocycle $\wp$ on $\T$, that is a map from
the set of oriented edges of $\T$ to $\Gr$ such that
\begin{enumerate}
\item the sum of the values of $\wp$ on the oriented edges forming the
  boundary of any face of $\T$ is zero and
\item $\wp(-e)=-\wp(e)$ for any oriented edge $e$ of $\T$, where $-e$ is $e$
  with opposite orientation.
\end{enumerate}
Each $\Gr$-coloring $\wp$ of $\T$ represents a cohomology class $[\wp]\in
H^1(M,G)$.

A {\it state} of a $\Gr$-coloring $\wp$ is a map $\p$ assigning to every
oriented edge $e$ of $\T$ an element $\p (e)$ of $I^{\wp(e)}$ such that
$\p(-e)=\p (e)^*$ for all $e$.  The set of all states of $\wp$ is denoted
$\states(\wp)$. The identities $\qd( {{\p(e)}})=\qd( {{\p(-e)}})$ and $\bb(
{{\p(e)}})=\bb( {{\p(-e)}})$ allow us to use the notation $\qd(\p(e))$ and
$\bb(\p(e))$ for non-oriented edges.

We call a $\Gr$-coloring of $(\T,\LL)$ \emph{admissible} if it takes values in
$\Gr\setminus \Gs$. Given an admissible $\Gr$-coloring $\wp$ of $(\T,\LL)$, we
define a certain partition function (state sum) as follows.  For each
tetrahedron $T$ of $\T$, we choose its vertices $v_1$, $v_2$, $v_3$, $v_4$ so
that the (ordered) triple of oriented edges
$(\vect{v_1v_2},\vect{v_1v_3},\vect{v_1v_4})$ is positively oriented with
respect to the orientation of $M$. Here by $\vect{v_1v_2}$ we mean the edge
oriented from $v_1$ to $v_2$, etc. For each $\p\in \states(\wp)$, set
$$
|T|_\p=\sjtop ijklmn
\text{ where }\left\{
\begin{array}{lll}
  i=\p(\vect{v_2v_1}),& j=\p(\vect{v_3v_2}),&
  k=\p(\vect{v_3v_1}),\\
  l=\p(\vect{v_4v_3}),&
  m=\p(\vect{v_4v_1}),& n=\p(\vect{v_4v_2}).
\end{array}\right.
$$
This $6j$-symbol belongs to the tensor product of $4$ multiplicity modules
associated to the faces of $T$ and does not depend on the choice of the
numeration of the vertices of $T$ compatible with the orientation of $M$. This
follows from the tetrahedral symmetry of the (modified) $6j$-symbol discussed
in Section \ref{S:Modified6j}.  Note that any face of $\T$ belongs to exactly
two tetrahedra of $\T$, and the associated multiplicity modules are dual to
each other. The tensor product of the $6j$-symbols $|T|_\p$ associated to all
tetrahedra $T$ of $\T$ can be contracted using this duality.  We denote by
$\cntr$ the tensor product of all these contractions.  Let $\T_1$ be the set
of unoriented edges $\T$ and let $\T_3$ the set of tetrahedra of $\T$. Set
$$
TV(\T,\LL,\wp)=\sum_{\p\in \states ( \wp)}\,\, \prod_{e\in\T_1{\setminus \LL}}
\qd(\p(e))\, \prod_{e\in \LL} \bb(\p(e))\, \, \cntr \left
  (\bigotimes_{T\in\T_3}|T|_{\p}\right ) \in K.
$$

\begin{theorem}\label{inveee} $
  TV(\T,\LL,\wp)$ depends only on the isotopy class of $L$ in $M$ and the
  cohomology class $[\wp]\in H^1(M,G)$. It does not depend on the choice of
  the $H$-triangulation of $(M,L)$ and on the choice of $\wp$ in its
  cohomology class.
\end{theorem}

A proof of this theorem will be given in the next section.

\begin{lemma}\label{eee+}
  Any $h\in H^1(M,G)$ can be represented by an admissible $G$-coloring on an
  arbitrary quasi-regular triangulation $\T$ of $M$.
\end{lemma}
\begin{proof}
  Take any $G$-coloring $\Phi$ of $\T$ representing $h$. We say that a vertex
  $v$ of $\T$ is {\it bad} for $\Phi$ if there is an oriented edge $e$ in $\T$
  outgoing from $v$ such that $\Phi(e)\in X$.  It is clear that $\Phi$ is
  admissible if and only if $\Phi$ has no bad vertices. We show how to modify
  $\Phi$ in its cohomology class to reduce the number of bad vertices.  Let $v
  $ be a bad vertex for $\Phi$ and let $E_v$ be the set of all oriented edges
  of $\T$ outgoing from $v$.  Pick any $$g\in G\setminus\left( \bigcup_{e\in
      E_v}(\Phi(e)+X)\right).$$ Let $c $ be the $G$-valued 0-cochain on $\T$
  assigning $g$ to $v $ and $0$ to all other vertices.  The 1-cocycle
  $\wp+\delta c$ takes values in $G\setminus X$ on all edges of $\T$ incident
  to $v$ and takes the same values as $\wp$ on all edges of $\T$ not incident
  to $v$. Here we use the fact that the edges of $\T$ are not loops which
  follows from the quasi-regularity of $\T$.  The transformation $\wp \mapsto
  \wp+\delta c$ decreases the number of bad vertices.  Repeating this
  argument, we find a 1-cocycle without bad vertices.
\end{proof}

We represent any $h\in H^1(M,G)$ by an admissible $G$-coloring $\wp$ of $\T$
and set
$$TV(M,L,h)=TV(\T,\LL,\wp)\in K.$$
By Theorem \ref{inveee}, $TV(M,L,h)$ is a topological invariant of the triple
$(M,L,h)$.

\subsection{Example}\label{SS:ExTVUslt} 
Set $\Gr=\C/2\Z$ and let $\cat$ be the pivotal tensor Ab-category with basic
data and t-ambi pair defined in Section \ref{SS:Uslt6j}. This category is
$\Gr$-graded: for $\wb\alpha\in\Gr$, the class $\cat_{\wb\alpha}$ consists of
the modules on which the central element $K^{\ro}\in\Uslt$ act as the scalar
$q^{{\ro}\wb\alpha}$. Set
$$\Gs =\Z/2\Z\subset\C/2\Z=\Gr.$$ It is easy to see that
$\Gr, \Gs$, and the constant function $\bb= {{\ro}^{-2}}$ satisfy the
requirements of Section \ref{SS:ModifiedTV} (cf.\ \eqref{E:TensorProdDe}).
The constructions above derive from this data a state-sum topological
invariant of links in $3$-manifolds.

The original Turaev-Viro-type state sum construction of 3-manifold invariants
can not be applied to $\cat$ because $\cat$ contains infinitely many
isomorphism classes of simple objects and is not semi-simple. Moreover, the
standard quantum dimensions of simple objects in $\cat$ are generically zero
and hence the standard $6j$-symbols associated with $\cat$ are generically
equal to zero.

\section{Proof of Theorem \ref{inveee}}\label{S:proofoftheoremeee}

Throughout this section, we keep notation of Theorem \ref{inveee}. We begin by
explaining that any two $H$-triangulations of $(M,L)$ can be related by
elementary moves adding or removing vertices, edges, etc.  We call an
elementary move positive if it adds edges and negative if it removes edges.

The first type of elementary moves are the so-called {\bubble} moves.  The
{\it positive {\bubble} move} starts with a choice of a face $F=v_1v_2v_3$ of
$\T$ such that at least one of its edges, say $v_2v_3$, is in $\LL$. Consider
two tetrahedra of $\T$ meeting along~$F$. We unglue these tetrahedra along $F$
and insert a $3$-ball between the resulting two copies of $F$. We triangulate
this $3$-ball by adding a vertex $v$ at its center and three edges $vv_1$,
$vv_2$, $vv_3$.  The edge $v_2v_3$ is removed from $\LL$ and replaced by the
edges $\{vv_2, vv_3\}$.
\begin{figure}[b]
  \centering \subfloat[{\bubble} move]{\label{F:bubble} \hspace{10pt} $
    \epsh{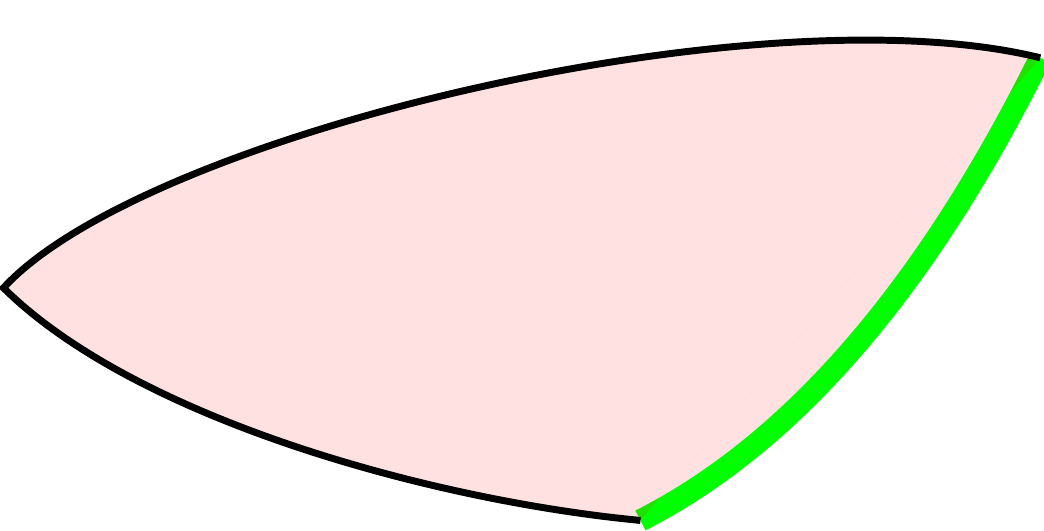}{30pt}\longrightarrow \epsh{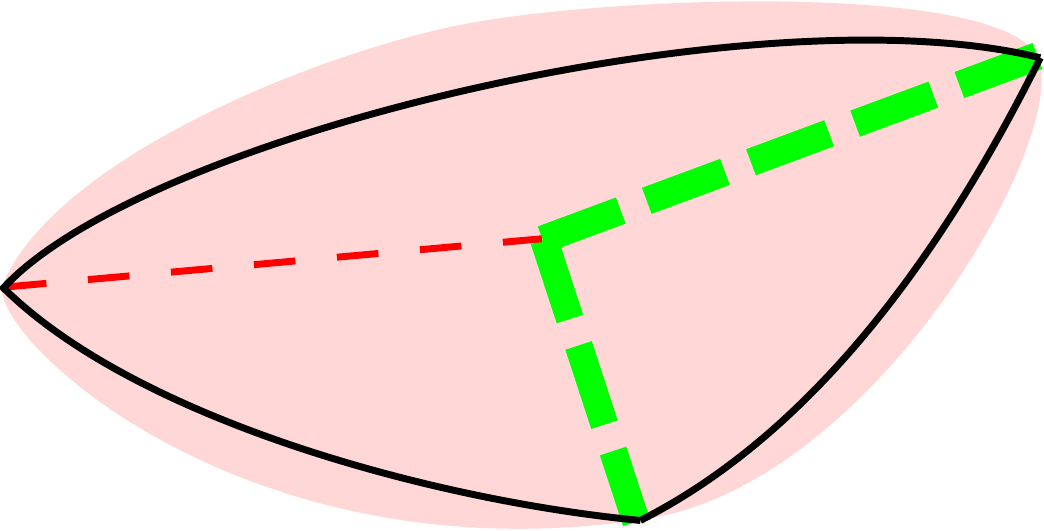}{30pt}
    \put(-23,0){\tiny v} $ \hspace{10pt} } \subfloat[{\Pachner}
  move]{\label{F:pachner} \hspace{10pt} $
    \epsh{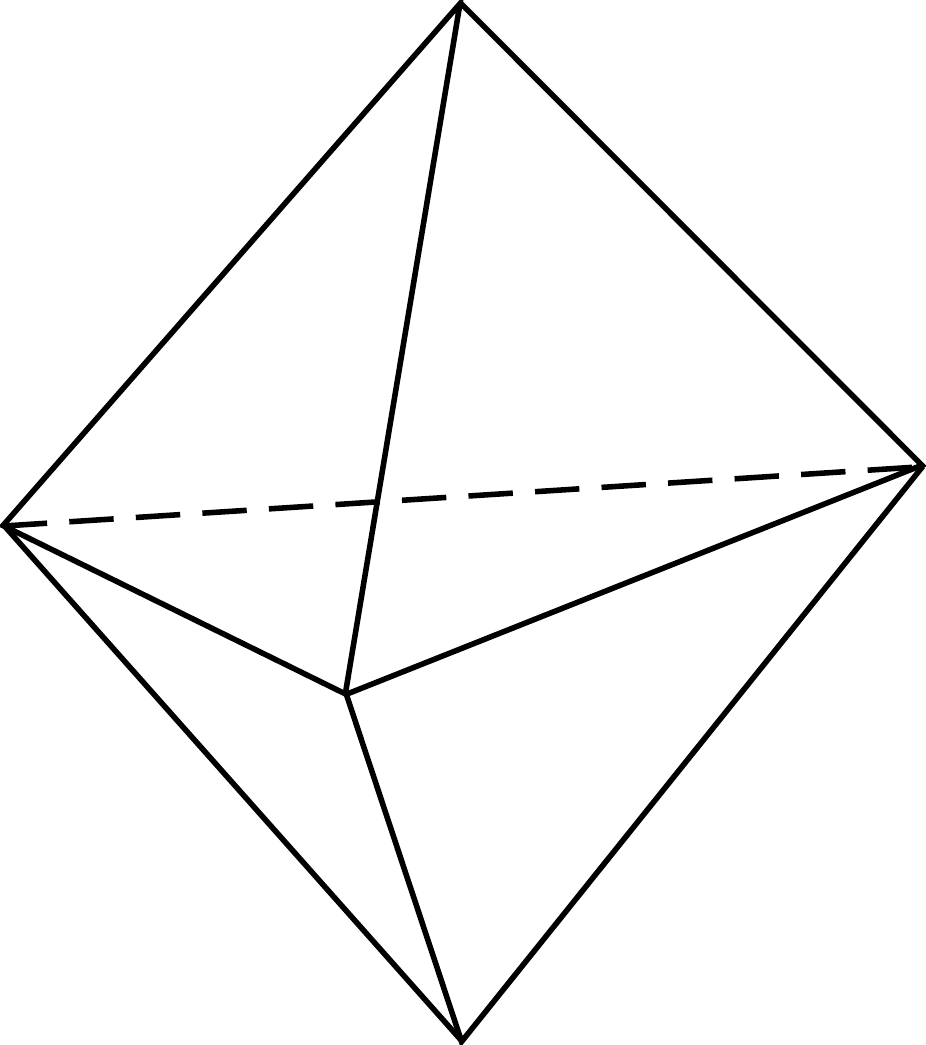}{50pt}\longleftrightarrow\epsh{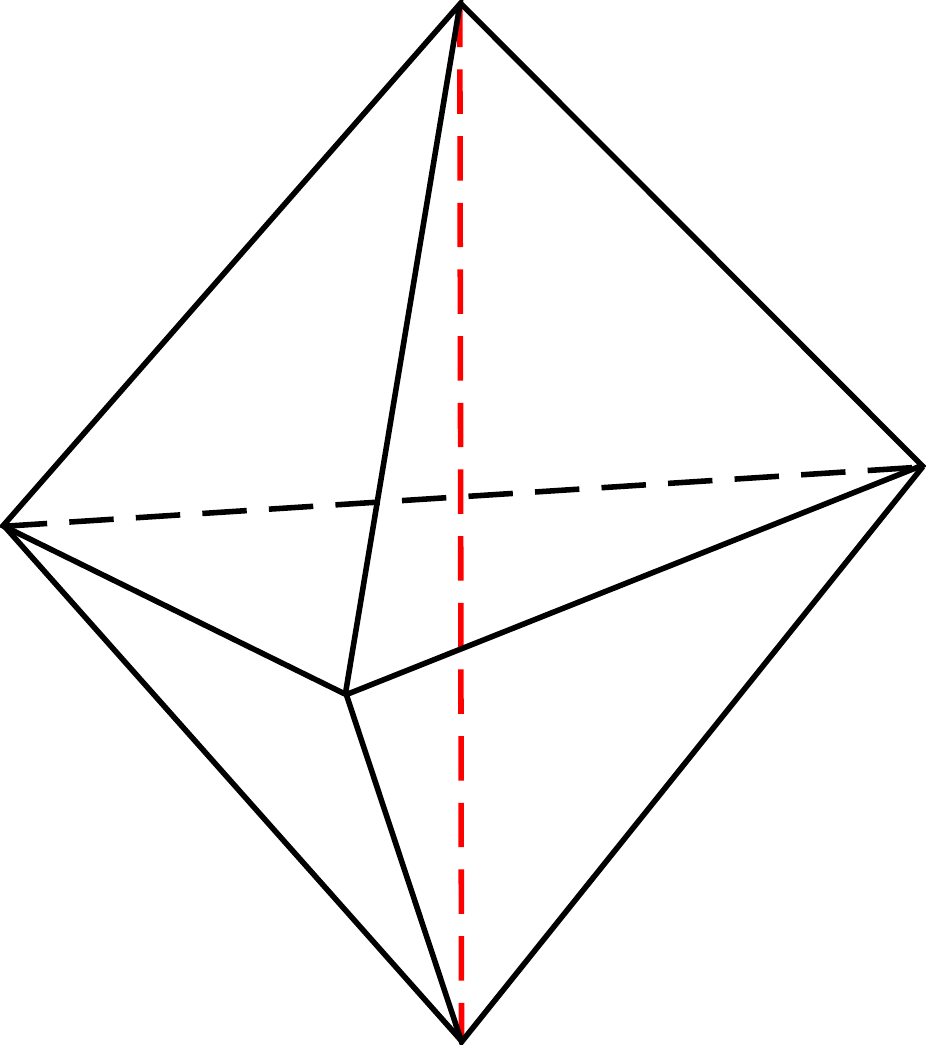}{50pt} $ \hspace{10pt}
  }\\
  \subfloat[{\lune} move]{\label{F:lune} \hspace{10pt}
    $\epsh{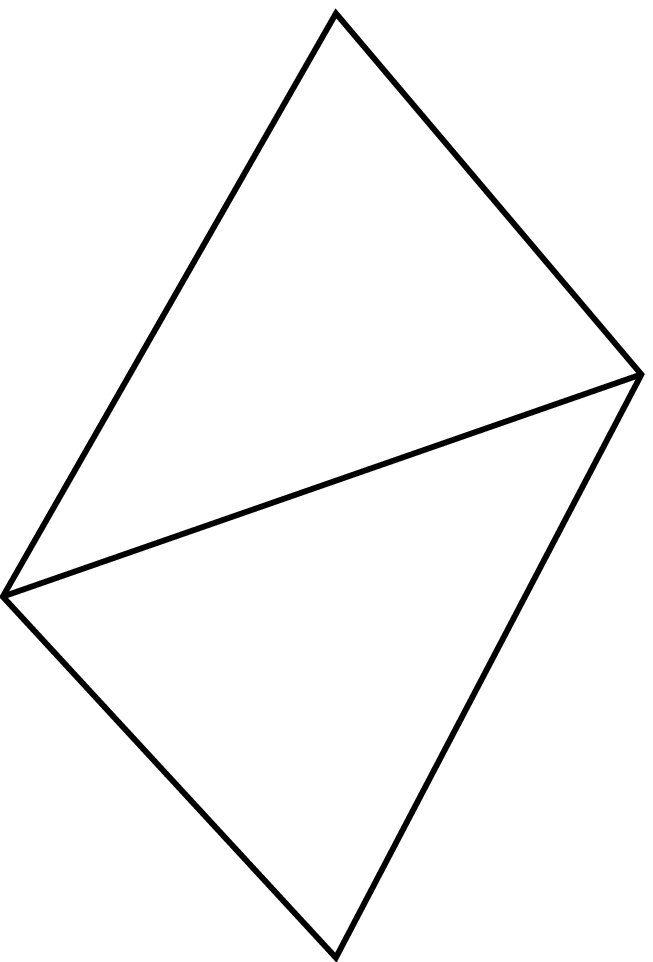}{50pt}\longleftrightarrow\epsh{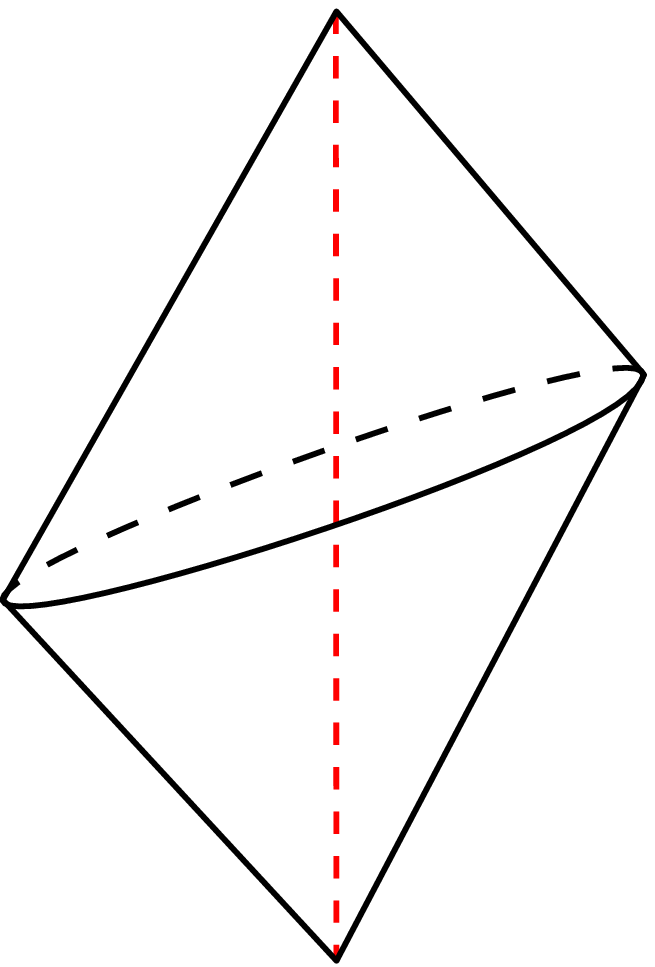}{50pt} $
    \hspace{10pt} }
  \caption{Elementary moves}
  \label{F:moves}
\end{figure}
This move can be visualized as in the transformation of Figure~\ref{F:bubble}
(where the bold (green) edges belong to $\LL$).  The inverse move is the {\it
  negative {\bubble} move}.

The second type of elementary moves is the {\it {\Pachner} $2\leftrightarrow
  3$ moves} shown in Figure~\ref{F:pachner}.  It is understood that the newly
added edge on the right is not an element of $\LL$. The negative {\Pachner}
move is allowed only when the edge common to the three tetrahedra on the right
is not in $\LL$.

\begin{proposition}[\cite{BB}, Proposition 4.23]\label{L:Toplemma}
  Let $L$ be a non-empty link in a closed connected orientable $3$-manifold
  $M$.  Any two $H$-triangulations of $(M,L)$ can be related by a finite
  sequence of {\bubble} moves and {\Pachner} moves in the class of
  $H$-triangulations of $(M,L)$.
\end{proposition}

We will need one more type of moves on $H$-triangulations of $(M,L)$ called
the {\it {\lune} moves}. It is represented in Figure \ref{F:lune} where for
the negative {\lune} move we require that the disappearing edge is not
in~$\LL$.  The {\lune} move may be expanded as a composition of {\bubble}
moves and {\Pachner} moves (see \cite{BB}, Section 2.1), but it will be
convenient for us to use the {\lune} moves directly.

The following   lemma is an algebraic analog of the {\bubble} move.

\begin{lemma}\label{L:algBubble}
  Let $g_1,g_2,g_3,g_4,g_5,g_6\in\Gr\setminus\Gs$ with $g_3=g_1+g_2$,
  $g_6=g_2+g_4$ and $g_5=g_1+g_6$.  If $ i\in I^{g_1}$, $j \in I^{g_2}$, $k\in
  I^{g_3}$, then
  \begin{align}\label{E:algLbubble}
    \qd(k) \sum_{\tiny \left.\begin{array}{c} l \in I^{g_4}, m\in I^{g_5} ,
          n\in I^{g_6}\\ \end{array}\right.}\hspace{-4ex} \qd(n) \bb(l)&\bb(m)
    *_{klm^*}*_{inm^*}*_{jln^*} \left( \sjtop
      ijklmn \otimes \sjtop k{j^*}inml \right) \notag \\
    &=\bb(k)\Id(i,j,k^*)
  \end{align}
\end{lemma}
\begin{proof}
  We can apply the orthonormality relation (Theorem \ref{T:orth}) to the tuple
  consisting of $i,j,k$, arbitrary $l \in I^{g_4},$ $m \in I^{g_5},$ and
  $p=k$.  Note that $k,m\in I_0$ and the pair $(i,j)$ is good because $
  V_i\otimes V_j\in \cat_{g_1+g_2}=\cat_{g_3}$ and $g_3\notin X$. Similarly,
  the pairs $ (j,l)$ and $ (k,l)$ are good.  Analyzing the grading in $\cat$,
  we observe that the set $N$ appearing in the orthonormality relation is a
  subset of $I^{g_6}$.   Moreover, if $n\in I^{g_6} \setminus N$ then 
  $$\sjtop ijklmn =0$$
  as $H(j,l,n^*)=0$ or $H(i,n,m^*)=0$ for such $n$.  
  If we multiply the orthogonality relation by $\bb(l)\bb(m)$, apply
  $*_{klm^*}$, and sum over all pairs $(l,m)\in I^{g_4}\times I^{g_5}$ we
  obtain that the left hand side of \eqref{E:algLbubble} is equal to
  \begin{equation}
  \label{E:bbII}
   \sum_{ (l,m)\in I^{g_4}\times I^{g_5} }
\bb(l)\bb(m) *_{klm^*}(\Id(i,j,k^*) \otimes \Id(k,l,m^*))
  \end{equation}
  Finally, using the equality $*_{klm^*}(\Id(k,l,m^*))=\dim(H(k,l,m^*))$ and
  the relations satisfied by~$\bb$, we obtain that the expression
  \eqref{E:bbII} is equal to $\bb(k)\Id(i,j,k^*)$.
 \end{proof}

\begin{lemma}\label{L:AdmMove}
  Let $\wp$ be an admissible $\Gr$-coloring of $\T$.  Suppose that
  $(\T',\LL')$ is an $H$-triangulation obtained from $(\T,\LL)$ by a negative
  {\Pachner}, {\lune} or {\bubble} move.  Then $\wp$ restricts to an
  admissible $\Gr$-coloring $\wp'$ of $\T'$ and
\begin{equation}\label{tvt}
  TV(\T,\LL,\wp)=TV(\T',\LL',\wp').
\end{equation}
\end{lemma}
\begin{proof} The values of $\wp'$ form a subset of the set of values of
  $\wp$, and therefore the admissibility of $\wp$ implies the the
  admissibility of $\wp'$.

  The rest of the proof is similar to the one in \cite[Section VII.2.3]{Tu}.
  In particular, we can translate the {\Pachner}, {\lune}, and {\bubble} move
  into algebraic identities: the Biedenharn-Elliott identity, the
  orthonormality relation and Equation \eqref{E:algLbubble}, respectively.
  The first two identities require certain pairs of indices to be good.  To
  see that all relevant pairs in this proof are good we make the following
  observation.  If $e_1,e_2, e_3$ are consecutive oriented edges of a $2$-face
  of $\T$, then $\wp(e_1)+\wp(e_2)= \wp(e_3)\in \Gr\setminus X$, and the
  semi-simplicity of $\cat_{\wp(e_3)}$ implies that the pair
  $(\p(e_1),\p(e_2))$ is good for any $\p\in \states (\wp)$. Translating into
  the language of $6j$-symbols, we obtain that all the $6j$-symbols involved
  in our state sums are admissible.

  We will now give a detailed proof of \eqref{tvt} for a negative {\bubble}
  move.  Let $\p'\in \states (\wp')$ and $S \subset \states(\wp)$ be the set
  of all states $\p$ of $\wp$ extending $\p'$.   It is enough to show that the
  term $TV_{\p'}$ of $TV(\T',\LL',\wp')$ associated to $\p'$ is equal to the sum $TV_S$ of the
  terms of $TV(\T,\LL,\wp)$ associated to the states in the set $S$.  Here it is equivalent to work with the positive {\bubble} move, which we do for convenience.

 Recall the description of the positive bubble given at the beginning of this section.  Let $ v,v_1, v_2, v_3$ (resp. $F=v_1v_2v_3$) be the vertices (resp. face) given in this description (see Figure \ref{F:col_bub}).
 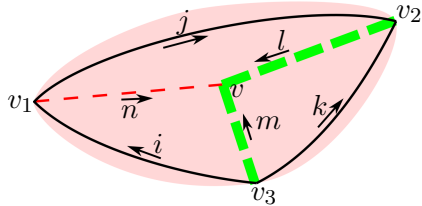
\begin{figure}[b]
 $\epsh{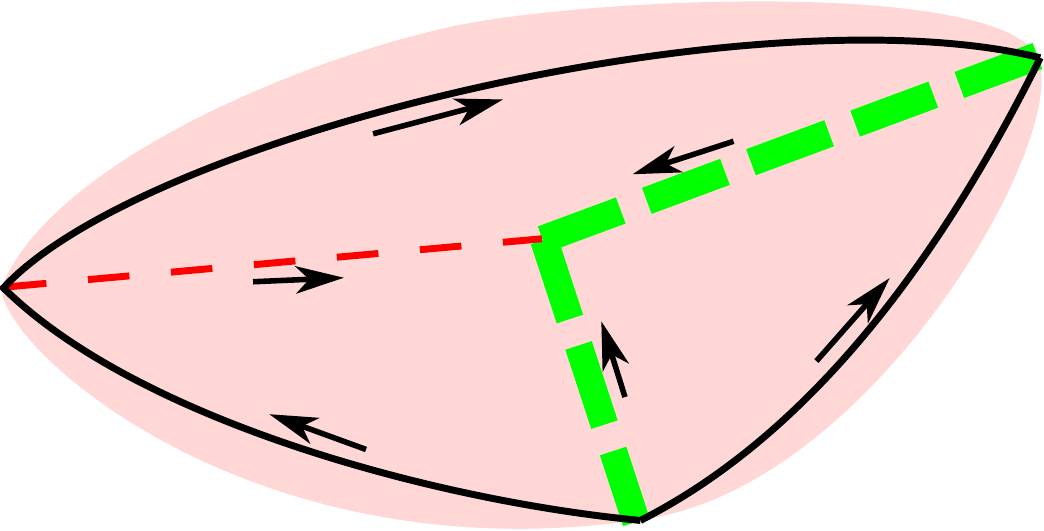}{70pt}$
 \put(-32,-5){\small $k$}\put(-53,-10){\small $m$}
 \put(-45,20){\small $l$}\put(-83,28){\small $j$}
 \put(-104,-6){\small $n$}\put(-92,-19){\small $i$}
 \put(-147,-2){\small $v_1$}\put(0,32){\small $v_2$}
 \put(-55,-37){\small $v_3$}
  \put(-63,2){\small $v$}
\caption{$T_1\cup T_2$ colored by $\p\in S$}
 \label{F:col_bub}
\end{figure}
Set $i=\p'(v_3v_1)$, $j=\p'(v_1v_2)$ and $k=\p'(v_3v_2)$. Let $T_1$ and $T_2$
be the two new tetrahedra of $\T$ and $f_1,f_2,f_3$ be their faces
$vv_2v_3,vv_3v_1,vv_1v_2$.  A state $\p\in S$ is determined by the values
$l,m,n$ of $\p$ on the edges $v_2v, v_3v, v_1v$, respectively (see Figure
\ref{F:col_bub}).

As explained above, in the bubble move two tetrahedra meeting along $F$ are
unglued and a $3$-ball is inserted between the resulting two copies $F_1$ and
$F_2$ of $F$. Note that $F_1$ and $F_2$ are faces of $\T$.  Also this $3$-ball
is triangulated with the two tetrahedra $T_1$ and $T_2$.  

For a fixed state on a triangulation, denote by $*_f$ the contraction along a
face $f$.  One can write $$TV_{\p'}=*_{F}(\bb(\p'(v_3v_2))X)$$ where $X$ is
all the factors of the state sum for $\p'$ except for $*_{F}$ and
$\bb(\p'(v_3v_2))$.  Since all $\p\in S$ restrict to $\p'$, we have that
$TV_S$ is equal to
$$
*_{F_1}*_{F_2}\left(X\otimes\sum_{\p\in
    S}{\mathsmal{\qd(\p(v_3v_2))\qd(\p(v_1v))\bb(\p(v_2v))\bb(\p(v_3v))}}
  *_{f_1}*_{f_2}*_{f_3} \left(|T_1|_\p\otimes|T_2|_\p\right)\right)
$$
where the additional factors come from the triangulation of the $3$-ball and
the fact that the link goes through $v_3v,vv_2$ instead of $v_3v_2$.  Here,
the contraction $*_{F_1}*_{F_2}$ is equal to $*_{ijk^*}*_{ijk^*}$.  On the
other hand, one has
$*_{F}(X)=*_{ijk^*}(X)=*_{ijk^*}*_{ijk^*}(X\otimes\Id(i,j,k^*))$.  Hence the
equality $TV_{\p'}=TV_S$ follows from

\begin{align*}
  \sum_{\p\in S}\qd(\p(v_3v_2))\qd(\p(v_1v))&\bb(\p(v_2v))\bb(\p(v_3v))
  *_{f_1}*_{f_2}*_{f_3} \left(|T_1|_\p\otimes|T_2|_\p\right) \\
  &=\Id(i,j,k^*)\bb(\p'(v_3v_2))
\end{align*}
which is exactly the identity established in Lemma~\ref{L:algBubble}.
\end{proof}

Let $\T_0$ be the set of vertices of $\T$.  Let $\delta$ be the coboundary
operator from the $G$-valued $0$-cochains on $\T$ to the $G$-valued
$1$-cochains on $\T$.

\begin{lemma}\label{L:Coboundary}
  Let $v_0\in\T_0$ and $c:\T_0\to \Gr$ be a map such that $c(v)=0$ for all
  $v\neq v_0$ and $c(v_0)\notin \Gs$.  If $\wp$ and $\wp+\delta c$ are
  admissible $\Gr$-colorings of $\T$, then
  $TV(\T,\LL,\wp)=TV(\T,\LL,\wp+\delta c)$
\end{lemma}

\begin{proof}
  In the proof, we shall use the language of skeletons of 3-manifolds dual to
  the language of triangulations (see, for instance \cite{Tu,BB}).  A skeleton
  of $M$ is a 2-dimensional polyhedron $P$ in $M$ such that $M\setminus P$ is
  a disjoint union of open 3-balls and locally $P$ looks like a plane, or a
  union of 3 half-planes with common boundary line in $\R^3$, or a cone over
  the 1-skeleton of a tetrahedron.  A typical skeleton of $M$ is constructed
  from a triangulation $T$ of $M$ by taking the union $P_T$ of the 2-cells
  dual to its edges. This construction establishes a bijective correspondence
  $T\leftrightarrow P_T$ between the quasi-regular triangulations $T$ of $M$
  and the skeletons $P$ of $M$ such that every 2-face of $P$ is a disk
  adjacent to two distinct components of $M-P$. To specify a Hamiltonian link
  $L$ in a triangulation $T$, we provide some faces of $P_T$ with dots such
  that each component of $M-P_T$ is adjacent to precisely two (distinct)
  dotted faces. These dots correspond to the intersections of $L$ with the
  2-faces. The notion of a $\Gr$-coloring on an $H$-triangulation $T$ of
  $(M,L)$ can be rephrased in terms of $P_T$ as a $2$-cycle on $P_T$ with
  coefficients in $\Gr$, that is a function assigning an
  element of $\Gr$ to every oriented 2-face of $P_T$ such that opposite
  orientations of a face give rise to opposite elements of $\Gr$ and the sum
  of the values of the function on three faces of $P_T$ sharing a common edge
  and coherently oriented is always equal to zero.  The notion of a state on
  $T$ can be also rephrased in terms of $P_T$. The state sum $TV(T, L, \wp)$
  can be rewritten in terms of $P=P_T$ in the obvious way, and will be denoted
  $TV(P,\wp)$ in the rest of the proof. The moves on the $H$-triangulations
  may also be translated to this dual language and give the well-known
  Matveev-Piergallini moves on skeletons adjusted to the setting of
  Hamiltonian links, see~\cite{BB}. We shall use the Matveev-Piergallini moves
  on dotted skeletons dual to the {\Pachner} move and to the {\lune}
  moves. Instead of the dual {\bubble} move we use the so-called $b$-move
  $P\to P'$. The $b$-move adds to a dotted skeleton $P\subset M$ a dotted
  2-disk $D\subset M$ such that the circle $\partial D$ lies on a dotted face
  $f$ of $P$, bounds a small 2-disk $D'\subset f$ containing the dot of $f$,
  and the 2-sphere $D\cup D'$ bounds an embedded 3-ball in $M$ meeting $P$
  solely along $D'$. Note that a dual $H$-bubble move on dotted skeletons is a
  composition of a $b$-move with a dual {\lune} move.

  We apply the $b$-move to the polyhedron $P=P_T$ as follows.  Consider the
  open 3-ball of $M-P_T$ surrounding the vertex $v_0$.  Assume first that the
  closure $ B$ of this open ball is an embedded closed 3-ball in $M$.  The
  2-sphere $\partial B$ meets~$L$ at two dots arising as the intersection of
  $P$ with the two edges of $L$ adjacent to~$v_0$. We call these dots the
  south pole and the north pole of $B$.  We apply the $b$-move $P\to P'=P\cup
  D$ at the south pole of $B$. Here $D\subset B$ is a 2-disk such that $D\cap
  P=\partial D=\partial D'$, where $D' $ is a small disk in $P$ centered at
  the south pole and contained in a face $f$ of $P$.  The given $\Gr$-coloring
  $\wp$ of $P$ induces a $\Gr$-coloring $\wp'$ of $P'$ which coincides with
  $\wp$ on the faces of $P$ distinct from $f$ and assigns to the faces $f-D'$,
  $D$, $D'$ of $P'$ the elements $\p(f)$, $g =c(v_0)\notin \Gs$, $\p(f)-g$ of
  $\Gr$, respectively. Here the orientation $D' $ is induced by the one of $M$
  restricted to $B$ and $f-D'$, $D$ are oriented so that $\partial
  D'= \partial D =- \partial (\overline{f-D'})$ in the category of oriented
  manifolds.  Next, we push the equatorial circle $\partial D$ towards the
  north pole of $\partial B$.  This transformation changes $P'$ by isotopy in
  $M$ and a sequence of Matveev-Piergallini moves dual to the {\Pachner} move
  and to the {\lune} moves. This is accompanied by the transformation of the
  $\Gr$-coloring $\wp'$ of $P'$ which keeps the colors of the faces not lying
  on $\partial B$ or lying on $\partial B$ to the north of the (moving)
  equatorial circle $\partial D$ and deduces $g$ from the $\wp$-colors for the
  faces lying on $\partial B$ to the south of $\partial D$. The color of the
  face $\Int D$ remains $g$ throughout the transformations.  These
  $\Gr$-colorings are admissible because all colors of the faces of the
  southern hemisphere are given by $\wp+\delta c$ whereas the colors of the
  faces of the northern hemisphere are given by $\wp$.  Hence, by Lemma
  \ref{L:AdmMove}, $TV(P',\wp')$ is preserved through this isotopy of
  $\partial D$ on~$\partial B$. Finally, at the end of the isotopy, $\partial
  D$ becomes a small circle surrounding the northern pole of $\partial B$.
  Applying the inverse $b$-move, we now remove $D$ and obtain the skeleton $P$
  with the $\Gr$-coloring $\wp+\delta(c)$. Hence $TV(P,\wp)=TV(P, \wp+\delta
  (c))$ which is equivalent to the claim of the lemma. In the case where the
  3-ball $B$ is not embedded in $M$, essentially the same argument applies.
  The key observation is that since $\T$ is a quasi-triangulation, the edges
  of $\T$ adjacent to $v_0$ are not loops, and therefore the ball $B$ does not
  meet itself along faces of $P$ (though it may meet itself along vertices
  and/or edges of $P$).
\end{proof}

\begin{lemma}\label{C:ConstCobound}   If $\wp$ and $\wp'$
  are two admissible $\Gr$-colorings of $\T$ representing the same class in
  $H^1(M;G)$, then $TV(\T,\LL,\wp)=TV(\T,\LL,\wp')$.
\end{lemma}
\begin{proof}
  As $\wp$ and $\wp'$ represent the same cohomology class, $\wp'=\wp+\delta
  c_1+\cdots+\delta c_n$ where $c_i:\T_0\to G$ is a 0-cochain taking non-zero
  value at a single vertex $ v_i$ for all $i=1,...,n$.  We prove the desired
  equality by induction on $n$.  If $n=0$ then $\wp'=\wp$ and the equality is
  clear.  Otherwise, let $E_{1}$ be the set of (oriented) edges of $\T$
  beginning at $v_1$. Pick any
  $$g\in \Gr\setminus\Big[ X\cup \bigcup_{e\in E_{1}} \big(\wp(e)+\Gs\big)
  \cup \bigcup_{e\in E_{1}} \big(c_1(v_1)+\wp'(e)+\Gs\big)\Big].$$ Let
  $c:\T_0\to G$ be the map given by $c(v_1)=g$ and $c(v)=0$ for all $v\neq
  v_1$.  Then $\wp+\delta c$ and $\wp+\delta c+\delta c_2+\cdots+\delta
  c_n=\wp'+ \delta (c-c_1)$ are admissible colorings.  Lemma
  \ref{L:Coboundary} and the induction assumption imply that
  \begin{align*}
  TV(\T,\LL,\wp)&=TV(\T,\LL,\wp+\delta c)\\
  &=TV(\T,\LL,\wp+\delta c+\delta c_2+\cdots+\delta c_n)\\
  &=TV(\T,\LL,\wp').
  \end{align*}
\end{proof}

\begin{theorem}\label{T:one-move}
  Let $(\T,\LL)$ and $(\T',\LL')$ be two $H$-triangulations of $(M,L)$ such
  that $(\T',\LL')$ is obtained from $(\T,\LL)$ by a single {\Pachner} move,
  {\bubble} move, or {\lune} move.  Then for any admissible $\Gr$-colorings
  $\wp$ and $\wp'$ on $\T$ and $\T'$ respectively, representing the same class
  in $H^1(M;G)$,
  $$
  TV(\T,\LL,\wp)=TV(\T',\LL',\wp').
  $$
\end{theorem}
\begin{proof}
  For concreteness, assume that $\T'$ is obtained from $\T$ by a negative
  move.  The admissible $\Gr$-coloring $\wp$ of $\T$ restricts to an
  admissible $\Gr$-coloring $\wp''$ of $\T'$ which represent the same class in
  $H^1(M;G)$.  Now Lemma~\ref{L:AdmMove} implies that
  $TV(\T,\LL,\wp)=TV(\T',\LL',\wp'')$ and Lemma \ref{C:ConstCobound} implies
  that $TV(\T',\LL',\wp'')=TV(\T',\LL',\wp')$.
\end{proof}

\begin{proof}[Proof of Theorem \ref{inveee}]
  From Proposition \ref{L:Toplemma} we know that any two $H$-triangulation of
  $(M,L)$ are related by a finite sequence of elementary moves.  Then the
  result follows from the Theorem \ref {T:one-move} by induction on the number
  of moves.
\end{proof}

\section{Totally symmetric $6j$-symbols}\label{S:totallySym}

Let $\cat$ be a pivotal tensor Ab-category with ground ring $K$, basic data
$\{V_i, w_i:V_i \to V_{i^*}^*\}_{i \in I}$, and t-ambi pair $(I_0,\qd)$.
Recall that the associated modified $6j$-symbols have the symmetries of an
oriented tetrahedron. The modified $6j$-symbols are totally symmetric if they
are invariant under the full group of symmetries of a tetrahedron. More
precisely, suppose that for every good triple $i,j,k\in I$, we have an
isomorphism $ \eta(i,j,k): H (i,j,k) \rightarrow H(k,j,i )$ satisfying the
following conditions:
$$\eta(i,j,k)=\eta( j,k,i)=\eta(k, i,j ),$$
$$ \eta(k,j,i) \circ \eta(i,j,k)={\text {id}} :H(i,j,k)\to H(i,j,k) $$
and
$$
(,)_{kji}\, (\eta(i,j,k)\otimes \eta(k^*,j^*,i^*))=(,)_{ijk}:H(i,j,k)\otimes_K
H(k^*, j^*, i^*)\to K
$$
where $ (,)_{ijk}$ is the pairing \eqref{E:pairing1-}.  We say that the
modified $6j$-symbols are \emph{totally symmetric} if for any good tuple
$(i,j,k,l,m,n)\in I^6$
$$
\sjtop ijklmn \circ \left[\eta(k^*, j,i)\otimes \eta( m^*,l, k)\otimes
  \eta(j^*,l^*,n)\otimes \eta(i^*, n^*,m)\right] =\sjtop jik{m^*}{l^*}{n^*}.
$$

For ribbon $\cat$, the associated modified $6j$-symbols are totally symmetric
(see \cite{Tu}, Chapter VI). The isomorphisms $\eta(i,j,k)$ in this case are
determined by so-called half-twists. A half-twist in $\cat$ is a family
$\{\theta'_i\in K\}_{i\in I}$ such that for all $i\in I$, we have
$\theta'_{i^*} =\theta'_i $ and the twist $V_i\to V_i$ is equal to
$(\theta'_i)^2 \, {\text {id}}_{V_i}$.

\begin{lemma}
  The modified $6j$-symbols defined in Section \ref{SS:Uslt6j} are totally
  symmetric.
\end{lemma}
\begin{proof}
  The category $\catd$ is ribbon with half-twist $\left(\twist'_i\right)_{i\in
    I}=\left(q^{(i/2)^2-(\m)^2}\right)_{i\in I}$, see \cite{GPT}.  This gives
  rise to a family of isomorphisms
  $$ \eta(i,j,k): H_{\catd}(i,j,k)
  \rightarrow H_{\catd} (k,j,i ) $$ making the modified $6j$-symbols
  associated with $\catd$ totally symmetric. Using the isomorphisms provided
  by Lemma \ref{L:CanIso}, one can check that this family induces a
  well-defined family of
  isomorphisms $$\eta(i,j,k):H_{\cat}(i,j,k) \rightarrow H_{\cat}(k,j,i).$$
  The latter family makes the modified $6j$-symbols associated with $\cat$
  totally symmetric.
\end{proof}

\begin{remark}
  For a category with totally symmetric $6j$-symbols, the construction of
  Section \ref{S:3Man} may be applied to links in non-oriented closed
  3-manifolds.
\end{remark}

\linespread{1}

\vfill

\end{document}